\documentclass{amsart}
\usepackage[utf8]{inputenc}
\usepackage{amsmath,amssymb}
\usepackage{tikz,tikz-cd}
\usepackage{comment}
\usepackage{ytableau}
\usepackage{enumerate}

\newtheorem{theorem}{Theorem}[section]
\newtheorem{lemma}[theorem]{Lemma}
\newtheorem{remark}[theorem]{Remark}
\newtheorem{proposition}[theorem]{Proposition}
\newtheorem{corollary}[theorem]{Corollary}

\newtheorem{definition}[theorem]{Definition}

\newcommand{\newword}[1]{\emph{\textbf{#1}}}
\newcommand{\kohnert}{\mathfrak{K}}
\newcommand{\fSchur}{\mathcal{S}^{\mathrm{flag}}}

\newcommand{\comp}[1]{\mathbf{#1}}
\newcommand{\wt}{\mathbf{wt}}
\newcommand{\sgn}{\mathbf{sign}}
\newcommand{\height}{\mathbf{ht}}

\newcommand{\D}{\mathbb{D}}
\newcommand{\schubert}{\mathfrak{S}}
\newcommand{\key}{\kappa}
\newcommand{\atom}{\mathcal{A}}

\newcommand{\h}{\mathfrak{h}}

\newcommand{\KD}{\mathrm{KD}}

\newcommand{\sh}{\mathbf{sh}}

\newcommand{\sort}{\mathbf{sort}}
\newcommand{\rev}{\mathbf{rev}}

\newcommand{\maj}{\mathbf{maj}}
\newcommand{\comaj}{\mathbf{comaj}}
\newcommand{\inv}{\mathbf{inv}}
\newcommand{\coinv}{\mathbf{coinv}}

\newcommand{\SSYT}{\mathrm{SSYT}}
\newcommand{\SSKT}{\mathrm{SSKT}}

\newcommand{\RSK}{\mathrm{RSK}}
\newcommand{\fRSK}{\mathrm{RSK}^{\mathrm{flag}}}
\newcommand{\rSSAF}{\overline{\mathrm{SSAF}}}
\newcommand{\rSSYT}{\overline{\mathrm{SSYT}}}

\newcommand{\pairs}{\mathcal{N}}
\newcommand{\fpairs}{\pairs^{\mathrm{flag}}}

\newlength\cellsize 

\savebox2{%
	\begin{picture}(10,10)
		\put(0,0){\line(1,0){10}}
		\put(0,0){\line(0,1){10}}
		\put(10,0){\line(0,1){10}}
		\put(0,10){\line(1,0){10}}
\end{picture}}

\newcommand\cellify[1]{\def\thearg{#1}\def\nothing{}%
	\ifx\thearg\nothing\vrule width0pt height\cellsize depth0pt%
	\else\hbox to 0pt{\usebox2\hss}\fi%
	\vbox to 10\unitlength{\vss\hbox to 10\unitlength{\hss$_{#1}$\hss}\vss}}

\newcommand\tableau[1]{\vtop{\let\\=\cr
		\setlength\baselineskip{-10000pt}
		\setlength\lineskiplimit{10000pt}
		\setlength\lineskip{0pt}
		\halign{&\cellify{##}\cr#1\crcr}}}

\newcommand\boxify[1]{\def\thearg{#1}\def\nothing{}%
	\ifx\thearg\nothing\vrule width0pt height\cellsize depth0pt%
	\else\hbox to 0pt{\usebox2\hss}\fi%
	\vbox to \cellsize{\vss\hbox to \cellsize{\hss$_{#1}$\hss}\vss}}

\savebox3{
	\begin{picture}(10,10)
		\put(5,5){\circle{10}}
\end{picture}}

\newcommand{\circify}[1]{\def\thearg{#1}\def\nothing{}%
	\ifx\thearg\nothing\vrule width0pt height\cellsize depth0pt%
	\else\hbox to 0pt{\usebox3\hss}\fi%
	\vbox to \cellsize{\vss\hbox to \cellsize{\hss$_{#1}$\hss}\vss}}

\newcommand\nullify[1]{\def\thearg{#1}\def\nothing{}%
	\ifx\thearg\nothing\vrule width0pt height\cellsize depth0pt%
	\else\hbox to 0pt{\hss}\fi%
	\vbox to \cellsize{\vss\hbox to \cellsize{\hss$_{#1}$\hss}\vss}}

\newcommand\cirtab[1]{\vline\vtop{\let\\=\cr
		\setlength\baselineskip{-8000pt}
		\setlength\lineskiplimit{8000pt}
		\setlength\lineskip{0pt}
		\halign{&\circify{##}\cr#1\crcr}}}

\newcommand\nulltab[1]{\vtop{\let\\=\cr
		\setlength\baselineskip{-8000pt}
		\setlength\lineskiplimit{8000pt}
		\setlength\lineskip{0pt}
		\halign{&\nullify{##}\cr#1\crcr}}}

\title{Complete Flagged Homogeneous Polynomials}

\date{}
\author{Henry Ehrhard}

\begin{document}
	
	\begin{abstract}
		We introduce a new basis for the polynomial ring which lifts the complete homogeneous symmetric polynomials while retaining representation theoretic significance. Using a specialized RSK algorithm we give an explicit nonnegative expansion into key polynomials and, generalizing the special rim hook tabloids of Eğecioğlu and Remmel, give an explicit signed expansion for key polynomials into this new basis.
	\end{abstract}

	\maketitle
	\tableofcontents
	
	%
	\section{Introduction}\label{sec:intro}
	%
	The classical theory of symmetric functions is beautifully described by the combinatorics of Young tableaux and has applications to the representation theory of the general linear group in particular. The distinguished symmetric basis $\{s_\lambda\}$ of Schur polynomials corresponds to irreducible representations. A related picture emerges when one considers subrepresentations of the Borel subgroup $B\subset \mathrm{GL}_n$ of lower triangular matrices. The characters $\key_\comp a$ are the key polynomials also called Demazure characters and they are no longer symmetric but rather a basis for the polynomial ring \cite{Dem74a,Dem74}. These objects remain of high combinatorial interest with many models that generalize the Schur polynomials. The key polynomials can be thought of as truncated, nonsymmetric, or ``flagged'' versions of the Schur polynomials (although ``flagged Schur polynomials'' refers to a different set of objects).
	
	The other classical symmetric bases have not found comparable counterparts, but we define a strong candidate for the complete homogeneous symmetric polynomials $\{h_\lambda\}$. We call them the complete flagged homogeneous polynomials $\h_\comp a$. These are a basis for the polynomial ring that lift the $\{h_\lambda\}$ basis in a suitable way. Like the Schur polynomials, the complete symmetric polynomials are characters of the general linear group so we might ask that $\h_\comp a$ is the character of a $B$-module, which pans out. This is to say that the new basis enjoys a similar relationship to the complete symmetric polynomials as the key polynomials do to Schur polynomials.
	
	On another axis, the relationship between $\{\key_\comp a\}$ and $\{ \h_\comp a \}$ refines the relationship between $\{ s_\lambda \}$ and $\{ h_\lambda \}$. The classical case is most easily encoded by the well-known Kostka coefficients, which describe the Schur expansion of $h_\lambda$ and are given combinatorial meaning by the RSK correspondence. We determine a similar model for the key expansion of $\h_\comp a$ by utilizing Haglund, Haiman, and Loehr's combinatorial formula for nonsymmetric Macdonald polynomials \cite{HHL08} and a nonsymmetric Cauchy identity discovered by \cite{Las03}. We reprove the latter by introducing an RSK analogue comparable to the one given by Mason \cite{Mas08}, but more specifically suited for the setting.
	
	The inverse Kostka coefficients for the opposite expansion were given a signed combinatorial formula by Eğecioğlu and Remmel \cite{ER90}. Here too, we obtain a satisfying model for the flagged analogues. The resulting combinatorics of snakes generalizes the notion of rim hooks, alternatively called ribbons or border strips. Rim hooks appear in other applications such as the Murnaghan-Nakayama rule and LLT polynomials \cite{Mur37,Nak40,LLT97}, so it is tempting to imagine that snakes could find additional use cases as well. We also discuss how the new expansion explains the cancellation that occurs in the classical formula.
	
	Our paper adheres to the following outline. In Section \ref{sec:h} we define the new basis and prove some basic facts. Section \ref{sec:rep} establishes a representation theoretic interpretation. In Section \ref{sec:nonsym-macdonald} we review the combinatorics of non-attacking fillings used in Haglund, Haiman, and Loehr's formula \cite{HHL08} and give our preferred interpretation of the flagged Kostka analogues. Section \ref{sec:rsk} develops our flagged RSK algorithm and is not logically necessary for understanding the rest of our paper. We show in Section \ref{sec:schubert} that the $\h_\comp a$ polynomials and, in fact, their products are nonnegative sums of Schubert polynomials. In Section \ref{sec:inverse} we define snakes as generalizations of rim hooks, and use them to describe the flagged inverse Kostka analogues.
	
	%
	\section{Complete Flagged Homogeneous Polynomials}\label{sec:h}
	%
	We begin with some basic notions. A \newword{weak composition} $\comp a=(\comp a_1,\comp a_2,\ldots,\comp a_n)$ of length at most $n$ is an element of $\mathbb N^n$, and we can always identify $\mathbb N^n$ with $\mathbb N^n\times\{0\}\subset\mathbb N^{n+1}$. The value $\comp a_i$ is called the $i$th \newword{part} of $\comp a$, and the sum of parts is the \newword{size} of $\comp a$. A \newword{partition} is a weak composition whose parts are in weakly decreasing order. Usually we will have an implicit fixed value of $n$ but we will sometimes be more explicit, particularly in Section \ref{sec:rsk}.
	
	Given a weak composition $\comp a$ and a set of indeterminates $x_1,x_2,\ldots, x_n$, we write $x^\comp a=\prod_{1\le i\le n} x_i^{\comp a_i}$. We let $\sort(\comp a)$ denote the partition obtained by putting the parts of $\comp a$ into weakly decreasing order, and $\rev(\comp a)$ denotes the weak composition $(\comp a_n,\comp a_{n-1},\ldots,\comp a_1)$.
	
	The \newword{complete homogeneous symmetric polynomial} of degree $k$ is
	\[
	h_k(x_1,\ldots,x_n)=\sum_{1\le i_1\le i_2\le\cdots\le i_k\le n} x_{i_1}x_{i_2}\cdots x_{i_k}.
	\]
	More generally, the complete homogeneous symmetric polynomial of a partition $\lambda$ is
	\begin{equation}
		h_{\lambda}(x_1,\ldots,x_n)=\prod_{1\le i\le n} h_{\lambda_i}(x_1,\ldots,x_n).
	\end{equation}
	These form a basis for the space of \newword{symmetric polynomials}. We let $\mathcal M_n$ denote the set of $n\times n$ matrices with entries in $\mathbb N$, and let $\mathcal L_n$ denote the subset of lower triangular matrices. Another well-known characterization of these polynomials is
	\begin{equation}
	h_{\lambda}(x_1,\ldots,x_n)=\sum_{\substack{M\in\mathcal M_n\\\mathrm{row}(M)=\lambda }} x^{\mathrm{col}(M)}
	\end{equation}
	where the \newword{row and column sums} $\mathrm{row}(M)$ and $\mathrm{col}(M)$ are the weak compositions whose $i$th part is the sum of entries in the $i$th row (respectively column) of $M$. The \newword{complete homogeneous symmetric function} indexed by the partition $\lambda$ is the stable limit of $h_\lambda(x_1,\ldots,x_n)$ as $n\to\infty$. We will typically use the symbol $h_\lambda$ alone to denote the stable limit, and specify the arguments when we want the polynomial. This notation makes more sense when one considers that the polynomial is recovered by setting all but the first $n$ indeterminates in $h_\lambda$ to zero.
	
	What follows is our nonsymmetric analogue of the complete homogeneous symmetric polynomials. 
	
	\begin{definition}\label{def:nonsym-h}
		The \newword{complete flagged homogeneous polynomial} indexed by a weak composition $\comp a$ is defined by
		\[
		\h_{\comp a}=\prod_{i\ge 1} h_{\comp a_i}(x_1,\ldots, x_i).
		\]
	\end{definition}
	
	Alternatively, we may think of these polynomials as sums over certain lower triangular matrices.
	\begin{lemma}\label{lem:hmatrix}
		We have \[
		\h_{\comp a}=\sum_{\substack{L\in\mathcal L_n\\\mathrm{row}(L)=\comp a}} x^{\mathrm{col}(L)}.
		\]
	\end{lemma}
	\begin{proof}
		We have
		\[
		\h_\comp a=\prod_{i\ge 1} h_{\comp a_i}(x_1,\ldots, x_i)=\prod_{i\ge 1}\sum_{\substack{\comp c\in\mathbb N^i\\|\comp c|=\comp a_i}}x^{\comp c}=\sum_{\substack{L\in \mathcal L_n\\\mathrm{row}(L)=\comp a}} x^{\mathrm{col}(L)},
		\]
		the last equality identifying a choice of $c\in \mathbb N^i$ with the $i$th row vector of a matrix.
	\end{proof}

	In order for us to consider these polynomials a convincing analogue for the complete symmetric basis, they should at minimum lift the complete homogeneous symmetric basis and themselves form a basis for polynomials. The next couple of propositions establish these facts.
	
	\begin{proposition}\label{prop:stable_h}
		The complete flagged homogeneous polynomials have the stable limit
		\[ \lim_{k\rightarrow\infty} \h_{0^k \times \comp{a}}(x_1,x_2,\ldots,x_k,0,0,\ldots) = h_{\sort(\comp{a})}.\]
	\end{proposition}
	\begin{proof}
		For a weak composition $\comp b$, the coefficient of $x^{\comp b}$ in $h_{\sort(\comp{a})}$ is the number of $\mathbb N$-matrices with row sums $\sort(\comp a)$ and column sums $\comp b$. This is the same as the number of $\mathbb N$-matrices with row sums $0^k \times \comp{a}$ and column sums $\comp b$. The coefficient of $x^{\comp b}$ in $\h_{0^k \times \comp{a}}(x_1,\ldots,x_k,0,\ldots,0)$ is the number of $\mathbb N$-matrices with row sums $0^k \times \comp{a}$, column sums $\comp b$, and support in the first $k$ columns. These objects coincide as soon as $k\ge n$.
	\end{proof}
	
	The \newword{dominance order} on weak compositions is defined by $\comp a\trianglelefteq \comp b$ if and only if $\comp a_1+\cdots+\comp a_k\le \comp b_1+\cdots+\comp b_k$ for all $k\ge 1$. In this case we say $\comp b$ \newword{dominates} $\comp a$.
	
	\begin{proposition}\label{prop:hbasis}
		The set $\{\h_{\comp{a}}\}$ indexed by weak compositions of length at most $n$ is a $\mathbb{Z}$-basis for $\mathbb{Z}[x_1,\ldots,x_n]$.
	\end{proposition}
	
	\begin{proof}
		Any nonzero entry in one of the first $k$ rows of a lower triangular $\mathbb N$-matrix $L$ is also contained in one of the first $k$ columns. Therefore $\mathrm{col}(L)\trianglerighteq\mathrm{row}(L)$. Using the characterization of $\h_{\comp a}$ in Lemma \ref{lem:hmatrix} we then have that $x^{\comp b}$ can appear in the monomial expansion of $\h_{\comp a}$ only if $\comp b$ dominates $\comp a$. Therefore the transition matrix from $\{\h_{\comp a}\}$ to $\{x^{\comp a}\}$ is upper triangular if both sets are ordered according to a linear extension of dominance order on the indexing compositions. In fact the transition matrix is upper uni-triangular since $x^{\comp a}$ appears in the expansion of $h_{\comp a}$ exactly once, indexed by the unique diagonal matrix with row sums $\comp a$. So the transition matrix is invertible. Since $\{x^{\comp a}\}$ is a basis for $\mathbb{Z}[x_1,\ldots,x_n]$, so is $\{\h_{\comp a}\}$.
	\end{proof}
	It is worth noting that the $\{\h_\comp a\}$ basis does not contain its symmetric counterpart. For example
	\[h_{11}(x_1,x_2)=\h_{02}+\h_{11}-\h_{20}.\]

	\section{$\h$ as a Character}\label{sec:rep}
	The polynomial $\h_{\comp a}$ also has a representation theoretic interpretation which we can see through the combinatorics of Kohnert polynomials. 
	
	We use the notation $[n]=\{1,2,\ldots,n\}$. A finite subset $D$ of $\mathbb Z_+\times [n]$ is called a \newword{diagram} and its elements are referred to as \newword{cells}. A cell is conceptualized as a box sitting at a lattice point in the first quadrant of the Cartesian coordinate plane. Thus, when we speak of a cell being ``weakly left'' of some other cell or column index, we are saying its first coordinate is weakly less than the comparable value. A cell lies ``strictly above'' row $r$ if its second index is strictly greater than $r$, etc.
	
	\begin{definition}\cite{Koh91}
		Given a diagram $D$ we may perform a \newword{Kohnert move} by removing the rightmost cell $(c,r)$ in some row of $D$, and appending a cell in the topmost vacant position in column $c$ strictly below row $r$, assuming such a position exists.
		
		The set of diagrams obtainable from $D$ through some sequence of Kohnert moves is denoted $\KD(D)$. Such diagrams are \newword{Kohnert diagrams} of $D$.
	\end{definition}

	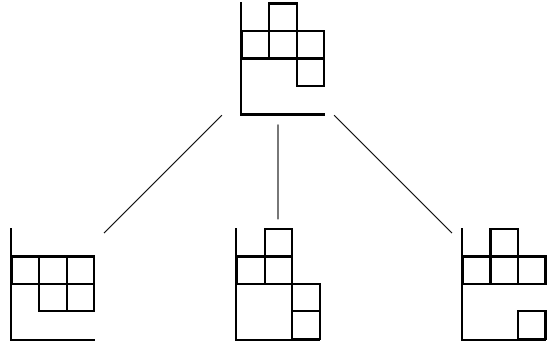
\begin{figure}[ht]
		\begin{center}
			\begin{tikzpicture}
				\node at (0,2) (A){
						\ytableausetup{centertableaux, boxsize=1em}
						\vline
						\begin{ytableau}
							\none&~&\none\\ ~ & ~&~\\ \none&\none&~\\\none&\none&\none \\\hline
						\end{ytableau}
				};
				\node at (-3,-1) (B){
					\vline
					\begin{ytableau}
						\none&\none&\none\\ ~ & ~&~\\ \none&~&~\\\none&\none&\none \\\hline
					\end{ytableau}
				};
				\node at (0,-1) (C) {
					\vline
					\begin{ytableau}
						\none&~&\none\\ ~ & ~&\none\\ \none&\none&~\\\none&\none&~ \\\hline
					\end{ytableau}
				};
				\node at (3,-1) (D){
					\vline
					\begin{ytableau}
						\none&~&\none\\ ~ & ~&~\\ \none&\none&\none\\\none&\none&~ 	\\\hline
					\end{ytableau}
				};
			\draw[thin] (A) -- (B);
			\draw[thin] (A) -- (C);
			\draw[thin] (A) -- (D);
			\end{tikzpicture}
		\end{center}
		\caption{\label{fig:kohnert}The three possible Kohnert moves on the top diagram.}
	\end{figure}

	\begin{definition}[\cite{AS22}]
		The \newword{Kohnert polynomial} of a diagram $D$ is
		\[
		\kohnert_D=\sum_{T\in\KD(D)}x^{\wt(T)}
		\]
		where $\wt(T)$ is the weak composition defined by $\wt(T)_i$ being the number of cells in row $i$ of $T$. By convention $\kohnert_\emptyset=1$.
	\end{definition}

	A diagram $D$ also indexes a $B$-module $\fSchur_D$ called a \newword{flagged Schur module}, where $B\subset\mathrm{GL}_n$ is the subalgebra of lower triangular matrices. The \newword{character} of $\fSchur_D$ is the trace of the action of the diagonal matrix $\mathrm{diag}(x_1,\ldots,x_n)$. Precise definitions can be found in \cite{RS98} or \cite{AABE23}.
	
	A diagram $D$ is \newword{southwest} if $(d,r),(c,s)\in D$ with $c<d$ and $r<s$ implies $(c,r)\in D$. The following was conjectured in \cite{AS22}.
	\begin{theorem}[\cite{AABE23}]\label{thm:kohnert}
		For a southwest diagram $D$ the character of $\fSchur_D$ coincides with the Kohnert polynomial $\kohnert_D$.
	\end{theorem}
	Thus, we can show that $\h_{\comp a}$ is the character of a flagged Schur module by showing it is the Kohnert polynomial of a southwest diagram.
	\begin{figure}[ht]
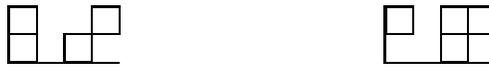

		\begin{displaymath}
			\ytableausetup{centertableaux}
			\vline
			\begin{ytableau}
				~&\none&\none&~\\ ~ & \none&~&\none \\\hline
			\end{ytableau}
			\hspace{100pt}
			\vline
			\begin{ytableau}
				~&\none&~&~\\ \none & \none& ~&\none \\\hline
			\end{ytableau}
		\end{displaymath}
		\caption{\label{fig:diagrams}A southwest diagram (left) and non-southwest diagram (right).}
	\end{figure}
	
	\begin{theorem}\label{thm:character}
		The polynomial $\h_{\comp a}$ is the character of the $B$-module $\fSchur_{D_{\comp a}}$ where
		\[
		D_{\comp a}=\left\{ (c,r)\in\mathbb \mathbb Z_+\times[n] \mid \sum_{k=1}^{r-1}\comp a_k < c\le \sum_{k=1}^{r}\comp a_k \right\}.\]
	\end{theorem}
	\begin{figure}[ht]
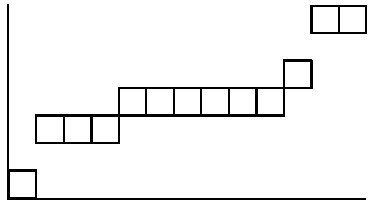

			\vline
			\begin{ytableau}
				\none&\none&\none&\none&\none&\none&\none&\none&\none&\none&\none&~&~\\
				\none\\
				\none&\none&\none&\none&\none&\none&\none&\none&\none&\none&~\\
				\none&\none&\none&\none&~&~&~&~&~&~\\
				\none&~&~&~\\
				\none \\
				~\\
				\hline
			\end{ytableau}
		\caption{\label{fig:bmod}The diagram $D_{(1,0,3,6,1,0,2)}$.}
	\end{figure}
	\begin{proof}
		The diagram is southwest as we see $(d,r),(c,s)\in D_{\comp a}$ with $r<s$ implies $d\le c$. Therefore it suffices to show $\h_{\comp a}=\kohnert_{D_{\comp a}}$ by Theorem \ref{thm:kohnert}. 
		
		We will show there is a bijection 
		\[
		\phi:\KD(D_{\comp a}) \overset{\sim}{\to} \{ L\in\mathcal L_n \mid \mathrm{row}(L)=\comp a \}
		\]
		such that the diagram weight corresponds to the column sums of the matrix. For $T\in\KD(D_{\comp a})$ we define 
		\[
		\phi(T)=\left( \left|\left\{ (c,r)\in T \mid r=j,\ \sum_{k=1}^{i-1}\comp a_k < c\le \sum_{k=1}^{i}\comp a_k  \right\} \right| \right)_{i,j}.
		\]
		There is exactly one cell in each column of $T$ up to $\sum_{k=1}^n \comp a_j$ so for fixed $i$ we have
		\[
		\sum_{j=1}^n \left|\left\{ (c,r)\in T \mid r=j,\ \sum_{k=1}^{i-1}\comp a_k < c\le \sum_{k=1}^{i}\comp a_k  \right\} \right|=\comp a_i.
		\]
		That is, $\mathrm{row}(\phi(T))=\comp a$. Additionally, for $(c,r)\in T$ with $\sum_{k=1}^{i-1}\comp a_k < c\le \sum_{k=1}^{i}\comp a_k$ we know that the only cell in column $c$ of $D_{\comp a}$ is $(c,i)$, which implies $r\le i$. Therefore $\phi(T)$ is always lower triangular, so $\phi$ is well-defined.
		
		Next for fixed $j$ we have
		\[
		\sum_{i=1}^n\left|\left\{ (c,r)\in T \mid r=j,\ \sum_{k=1}^{i-1}\comp a_k < c\le \sum_{k=1}^{i}\comp a_k  \right\} \right|=\left|\left\{ (c,r)\in T \mid r=j\right\} \right|=\wt(T)_j
		\]
		so $\phi$ translates weights to column sums as intended.
		
		We claim that any $T\in\KD(D_{\comp a})$ has the property that if 
		\[
		\left( \sum_{k=1}^{r-1}\comp a_k ,b_1 \right),\left( \sum_{k=1}^{r-1}\comp a_k+1 ,b_2 \right),\ldots,\left( \sum_{k=1}^{r}\comp a_k ,b_{\comp a_r} \right)
		\]
		are the (unique) cells in these columns then $b_1\ge b_2\ge\cdots\ge b_{\comp a_r}$. This is true for $D_{\comp a}$ where $r=b_1=\cdots=b_{\comp a_r}$. Supposing the property holds for some $T\in \KD(D_{\comp a})$, a Kohnert move can only replace a cell $\left( \sum_{k=1}^{r-1}\comp a_k+i ,b_i \right)$ with $\left( \sum_{k=1}^{r-1}\comp a_k+i ,b_i-1 \right)$ since there is only one cell in each column. Even then, this is only possible if there is no cell in the same row to the right of $\left( \sum_{k=1}^{r-1}\comp a_k+i ,b_i \right)$, which is to say $b_i>b_{i+1}$ when $i<\comp a_r$. The property $b_1\ge b_2\ge\cdots\ge b_{\comp a_r}$ then must be maintained under Kohnert moves.
		
		The above implies that for $L\in\mathcal L_n$ with $\mathrm{row}(L)=\comp a$ there is only one diagram of which $L$ can possibly be the image. Namely, we must have that within the columns $\sum_{k=1}^{r-1} \comp a_k,\ldots, \sum_{k=1}^r \comp a_k$ the rightmost $L_{r,1}$ columns must have cells in row $1$, the next rightmost $L_{r,2}$ columns have cells in row $2$, and so on to $L_{r,r}$.
		
		For any such $L$ the diagram just described is a Kohnert diagram. Indeed, starting with the diagram $D_{\comp a}$ and $r=1$, within the columns $\sum_{k=1}^{r-1} \comp a_k,\ldots, \sum_{k=1}^r \comp a_k$ we use Kohnert moves to put the rightmost put the cells in the rightmost $L_{r,1}$ columns into row 1, followed by the next rightmost $L_{r,2}$ columns into row 2, etc. We repeat this procedure for $r=2$ then $r=3$ and so on. These are all valid sequences of Kohnert moves because the cells in columns $\sum_{k=1}^{r-1} \comp a_k,\ldots, \sum_{k=1}^r \comp a_k$ start in a row strictly below any other cell to the right.
		
		Therefore $\phi$ is a bijection which allows us to conclude
		\[
		\kohnert_{D_{\comp a}}=\sum_{\substack{L\in\mathcal L_n\\\mathrm{row}(L)=\comp a}} x^{\mathrm{col}(L)}=\h_{\comp a}
		\]
		by Lemma \ref{lem:hmatrix}.
	\end{proof}

	It follows as a corollary that $\h_{\comp a}$ is a positive sum of key polynomials defined in Section \ref{sec:nonsym-macdonald}. Definitions of peelable tableaux and Yamanouchi Kohnert diagrams can be found in \cite{RS98} and \cite{AABE23} respectively.

	\begin{corollary}
		We have
		\[
		\h_{\comp b}=\sum_{\comp a} \tilde K_{\comp{a}\comp{b}} \key_{\comp a} 
		\]
		where $\tilde K_{\comp a\comp b}$ are non-negative integers. In particular, $\tilde K_{\comp a\comp b}$ is the number of $D_{\comp b}$-peelable tableaux whose left key is $\comp a$, and the number of Yamanouchi Kohnert diagrams of $D_{\comp b}$ with weight $\comp a$.
	\end{corollary}
	\begin{proof}
		Since $\h_{\comp b}$ is the character of $\fSchur_{D_{\comp b}}$ with $D_{\comp b}$ southwest, the two combinatorial interpretations for the coefficients follow immediately from \cite[Thm~20]{RS98} and \cite[Cor~4.1.3]{AABE23} respectively.
	\end{proof}

	In Section \ref{sec:nonsym-macdonald} we provide another combinatorial interpretation for the numbers $\tilde K_{\comp a\comp b}$ that is more closely analogous to the usual model for Kostka coefficients.
	
	%
	\section{Keys, Atoms, and Nonsymmetric Macdonald Polynomials}\label{sec:nonsym-macdonald}
	%
	Macdonald's symmetric functions $P_{\lambda}(x;q,t)$ \cite{Mac88} are symmetric functions indexed by partitions $\lambda$ with rational coefficients in $q,t$. They interpolate the Hall--Littlewood symmetric functions $H_{\lambda}(x;t) = P_{\lambda}(x;0,t)$ and the Jack symmetric functions $J_{\lambda,\alpha}(x) = \lim_{t \rightarrow 1} P_{\lambda}(x;t^{\alpha},t)$. They further specialize to the classical bases \[m_\lambda = P_\lambda(x;q,1),\, s_\lambda = P_{\lambda}(x;q,q), \text{ and }e_{\lambda'}=P_\lambda (x;1,t)\]
	regardless of the values of the free parameters.
	
	The nonsymmetric Macdonald polynomials $E_a(x_1,\ldots,x_n;q,t)$ were introduced by Opdam \cite{Opd95} and Macdonald \cite{Mac96}. They recover their symmetric analogues by
	\begin{equation}
		P_{\lambda}(x_1,\ldots,x_n;q,t)=\frac{\prod_{u\in\D(\comp a)}1-q^{\mathrm{leg}(u)+1}t^{\mathrm{arm}(u)+1}}{\prod_{u\in\D(\rev(\lambda))}1-q^{\mathrm{leg}(u)}t^{\mathrm{arm}(u)+1}}E_{0^n\times\comp a}(x_1,\ldots,x_n;q,t)
	\end{equation}
	where $\comp a$ is any rearrangement of $\lambda$, and $\mathrm{arm}$ and $\mathrm{leg}$ are defined below \cite[Cor~5.2.2]{HHL08}. Haglund, Haiman and Loehr gave a combinatorial formula for the monomial expansion of nonsymmetric Macdonald polynomials as follows \cite{HHL08}.
	
	Given a weak composition $\comp a$, its \newword{key diagram} is 
	\begin{equation}
		\D(\comp a)=\{ (c,r)\in\mathbb Z_+\times[n] \mid c\le \comp a_r \}.
	\end{equation} 
	A \newword{filling} of a diagram is a map from the diagram to $[n]$. The \newword{weight} $\wt(T)$ of a filling $T$ is the weak composition defined by $\wt_i(T)=|T^{-1}(i)|$. When $T$ is a filling of a key diagram the \newword{shape} $\sh(T)$ of $T$ is the corresponding weak composition.
	 
	Given a filling $T:\D(\comp a)\to\mathbb [n]$ of $\D(\comp a)$, define the \newword{augmented filling} $\hat T$ on $\hat\D(\comp a)=\D(\comp a)\sqcup\{ (0,1),\ldots,(0,n) \}$ by
	\[
	\hat T(u)=\begin{cases}
		T(u) &\text{if } u\in\D(\comp a)\\
		i&\text{if } u=(0,i)
	\end{cases}.
	\]
	The cells $\{ (0,1),\ldots,(0,n) \}$ are called the \newword{basement}.
	
	A pair of cells \newword{attack} each other if they lie in the same column or in adjacent columns with the cell on the left strictly higher than the cell on the right. A filling $T$ is \newword{non-attacking} if no pair of attacking cells have the same value in $T$.
	
	The \newword{arm} $\mathrm{arm}(u)$ of a cell $u\in\D(\comp a)$ is the number of cells below $u$ in the same column plus the number of cells above $u$ in the left adjacent column of $\hat\D(\comp a)$. The \newword{leg} $\mathrm{leg}(u)$ of a cell $u\in\D(\comp a)$ is the number of cells weakly to its right in the same row. Given a non-attacking filling $T$, the \newword{major index} of $T$, denoted by $\maj(T)$, is the sum of the legs of all cells $u$ such that $T(u)$ is strictly greater than the entry of the cell immediately to its left. Similarly, the \newword{co-major index}, denoted by $\comaj(T)$, is the sum of the legs of all cells $u$ such that $T(u)$ is strictly less than the entry immediately left of $u$. The conditions $\maj(T)=0$ and $\comaj(T)=0$, which will concern us, respectively say that the entries of $T$ are weakly decreasing and increasing within rows from left to right.
	
	A \newword{triple} of a non-attacking filling $T$ of $\D(\comp a)$ is three cells with distinct entries of either the form
	\begin{description}
		\item[Type I] $(r,c),(r,c+1),(s,c)\in\D(\comp a)$ with $r<s$ and $\comp a_r>\comp a_s$, or
		\item[Type II] $(r,c+1),(s,c),(s,c+1)\in\D(\comp a)$ with $r<s$ and $\comp a_r\le \comp a_s$.
	\end{description}
	Say that $i,j,k$ are the entries of a triple where $k$ is left of $i$ in the same row. The triple is a \newword{co-inversion triple} if $i<j<k$ or $j<k<i$ or $k<i<j$. An \newword{inversion triple} satisfies one of the reverse inequalities. The Type I and Type II co-inversion triples are illustrated in Fig.~\ref{fig:inv}. Let $\coinv(T)$ (resp. $\inv(T)$) denote the number of co-inversion (resp. inversion) triples of $T$.

	\begin{figure}[ht]
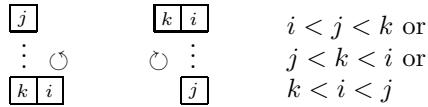

		\begin{displaymath}
			\begin{array}{l}
				\tableau{ j } \\[-0.5\cellsize] \hspace{0.4\cellsize} \vdots \hspace{0.5\cellsize} \circlearrowleft \\ \tableau{ k & i }
			\end{array}\hspace{2\cellsize}
			\begin{array}{r}
				\tableau{ k & i } \\[-0.5\cellsize] \circlearrowright \hspace{0.5\cellsize} \vdots \hspace{0.4\cellsize} \\ \tableau{ j }
			\end{array}\hspace{2\cellsize}
			\begin{array}{l}
				i < j < k \text{ or} \\
				j < k < i \text{ or} \\
				k < i < j
			\end{array}
		\end{displaymath}
		\caption{\label{fig:inv}The positions and orientation for co-inversion triples.}
	\end{figure}

	\begin{remark}
		Note that our definition of inversion triples is \textit{not} equivalent to the one used in \cite{HHL08} and elsewhere. The authors allow their inversion triples to contain repeat entries and use $\coinv'$ to denote the statistic we call $\inv$.
	\end{remark}

	\begin{remark}\label{rem:inv}
		We will only care about co-inversion triples when the major index is zero, and inversion triples when the co-major index is zero. In these cases, the only realizable inequalities given one of the arrangements in Figure \ref{fig:inv} are $i<j<k$ and $i>j>k$ respectively.
	\end{remark}
	
	The combinatorial identity is now as follows.
	\begin{theorem}[\cite{HHL08}]\label{thm:HHL}
		The nonsymmetric Macdonald polynomial is given by
		\[ E_{a}(x;q,t) =\hspace{-0.3cm} \sum_{\substack{T:\D(\comp a)\rightarrow [n] \\ \hat T\mathrm{non-attacking}}}\hspace{-0.3cm} q^{\maj(\hat T)} t^{\coinv(\hat T)} x^{\wt(T)} \hspace{-0.5cm}
		\prod_{\substack{u\in \D(\comp a)\\ \hat T(u) \neq \hat T(\mathrm{left}(u))}} \hspace{-0.5cm} \frac{1-t}{1 - q^{\mathrm{leg}(u)+1} t^{\mathrm{arm}(u)+1}} . \]
		
		where $\mathrm{left}(u)$ is the cell immediately left of $u$ in $\hat D(\comp a)$. Additionally,
		\[ E_{a}(x;q^{-1},t^{-1}) =\hspace{-0.3cm} \sum_{\substack{T:\D(\comp a)\rightarrow [n] \\ \hat T\mathrm{non-attacking}}} \hspace{-0.3cm} q^{\comaj(\hat T)} t^{\inv(\hat T)} x^{\wt(T)} \hspace{-0.5cm}
		\prod_{\substack{u\in \D(\comp a)\\ \hat T(u) \neq \hat T(\mathrm{left}(u))}}\hspace{-0.5cm} \frac{1-t}{1 - q^{\mathrm{leg}(u))+1} t^{\mathrm{arm}(u)+1}} . \]
		\label{thm:mac}
		
	\end{theorem}

	The \newword{Ferrers diagrams} of a partition $\lambda$ are $\D(\lambda)$ and $\D(\rev(\lambda))$. The former is known as the French convention and the latter is the English convention. We use whichever is convenient.
	
	\begin{definition}
		A \newword{semi-standard Young tableau} (SSYT) is a filling of a Ferrers diagram in the French convention such that column entries strictly increase bottom to top and row entries weakly increase left to right. The set of SSYT is denoted $\SSYT$.
	\end{definition}

	This leads to the combinatorial definition of \newword{Schur polynomials}:
	\begin{equation*}
		s_\lambda(x_1,\ldots, x_n)=\sum_{\comp a\in\mathbb N^n} K_{\lambda\comp a}x^{\comp a}
	\end{equation*}
	where
	\[
	K_{\lambda\comp a}=|\{ Q\in\SSYT \mid \sh(Q)=\lambda, \wt(Q)=\comp a  \}|
	\]
	are the classical \newword{Kostka coefficients}. It is well-known that $K_{\lambda\comp a}$ depends only on $\lambda$ and $\sort(\comp a)$. The stable limit of a Schur polynomial as $n\to\infty$ is the \newword{Schur function} $s_\lambda$ with notation completely analogous to $h_\lambda$. The Schur polynomials and complete homogeneous symmetric polynomials are related by the expansion
	\[
	h_\mu(x_1,\ldots,x_n) = \sum_{\lambda} K_{\lambda\mu}  s_\lambda (x_1,\ldots,x_n)
	\]
	summed over partitions. This is a consequence of the RSK algorithm which is reviewed in Section \ref{sec:rsk}.
	
	The Schur polynomials are generalized by the \newword{key polynomials} $\key_{\comp a}$ indexed by weak compositions. These were first introduced by Demazure and are also known as Demazure characters \cite{Dem74}. They coincide with the characters of $\fSchur_{D(\comp a)}$. A simple consequence of Sanderson's relation between Macdonald polynomials and affine Demazure characters \cite{San00} is that nonsymmetric Macdonald polynomials specialize to key polynomials, which was also shown combinatorially by Assaf \cite{Ass18}.
	
	\begin{theorem}[\cite{San00}]\label{thm:key_mac}
		The key polynomial $\key_\comp a(x)$ is given by
		\begin{equation}\label{eq:key}
			\key_\comp a(x) = E_{\comp a}(x;0,0).
		\end{equation}
		
	\end{theorem}

	A reverse partition yields the Schur polynomial:
	\begin{equation}
		s_\lambda(x_1,\ldots, x_n)=\key_{\rev(\lambda)}.
	\end{equation}
	
	Let $t_{i,j}$ denote the transposition exchanging $i$ and $j$ with $i<j$. For a permutation $w=w_1\cdots w_k$ let $\ell(w)$ denote the number of pairs $i<j$ such that $w_i>w_j$. \newword{Bruhat order} on the symmetric group is defined to be the transitive closure of the cover relation $w\le wt_{i,j}$ whenever $\ell(w t_{i,j}) = \ell(w) + 1$.
	
	The \newword{Demazure atoms} $\atom_\comp a(X)$ were originally studied under the name standard bases by Lascoux and Schützenberger \cite{LS90}. They are a monomial positive $\mathbb Z$-basis for polynomials characterized by
	\[
	\key_{\comp a}=\sum_{\comp b\le \comp a} \atom_{\comp b}
	\]
	where $\comp b\le\comp a$ means there exists a partition $\lambda$ and permutations $\sigma\le \tau$ such that $\comp b_i=\lambda_{\sigma(i)}$ and $\comp a_i=\lambda_{\tau(i)}$ for all $i$. See \cite{P16} for an overview. Mason showed that the Demazure atoms are also specializations of nonsymmetric Macdonald polynomials.
	
	\begin{theorem}[\cite{Mas09}]\label{thm:atom_mac}
		The Demazure atom $\atom_\comp a(x)$ is given by
		\begin{equation}\label{eq:atom}
			\atom_a(x) = E_{\rev(\comp a)}(x_n,x_{n-1},\ldots,x_1;\infty,\infty).
		\end{equation}
	\end{theorem}

	Theorem \ref{thm:HHL} along with the specializations \eqref{eq:key} and \eqref{eq:atom} motivate the following definitions.
	
	\begin{definition}[\cite{Ass18}]
		A \newword{semi-standard key tableau} (SSKT) is a filling $T$ of a key diagram such that $\hat T$ is non-attacking and $\maj(\hat T)=\coinv(\hat T)=0$. The set of SSKT is denoted $\SSKT$.
	\end{definition}
	
	\begin{definition}
		A \newword{reverse semi-skyline augmented filling} (reverse SSAF) is a filling $T$ of a key diagram such that $\hat T$ is non-attacking and $\comaj(\hat T)=\inv(\hat T)=0$. The set of reverse SSAF is denoted $\rSSAF$.
	\end{definition}

	\begin{figure}[ht]
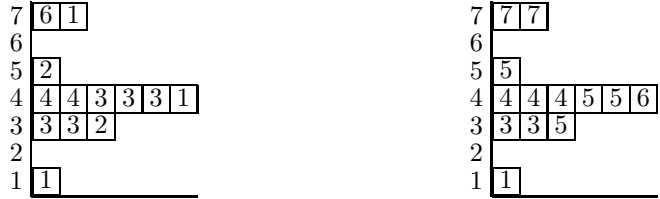

		\begin{displaymath}
			\ytableausetup{centertableaux,boxsize=1em}
			\begin{ytableau}
				\none[7]\\\none[6]\\\none[5]\\\none[4]\\\none[3]\\\none[2]\\\none[1]\\
			\end{ytableau}
			\vline
			\begin{ytableau}
				6&1\\
				\none \\
				2\\
				4&4&3&3&3&1\\
				3&3&2\\
				\none\\
				1\\
				\hline
			\end{ytableau}
			\hspace{100pt}
			\begin{ytableau}
				\none[7]\\\none[6]\\\none[5]\\\none[4]\\\none[3]\\\none[2]\\\none[1]\\
			\end{ytableau}
			\vline
			\begin{ytableau}
				7&7\\
				\none \\
				5\\
				4&4&4&5&5&6\\
				3&3&5\\
				\none\\
				1\\
				\hline
			\end{ytableau}
		\end{displaymath}
		\caption{\label{fig:fillings}An SSKT (left) and a reverse SSAF (right) of shape $(1,0,3,6,1,0,2)$.}
	\end{figure}

	The reverse SSAF are equivalent to Mason's semi-skyline augmented fillings up to reversal of the alphabet and shape \cite{Mas08}.
	
	What follow are some simple facts about SSKT that will find their use in later sections.
	\begin{lemma}\label{lem:ledge}
		If row $r$ of an SSKT is weakly shorter than row $s$ with $r<s$ then any entry in row $r$ is strictly less than the entry in the same column of row $s$.
	\end{lemma}
	\begin{proof}
		This is true in the basement column. Suppose in some column we have entries $i<j$ in rows $r,s$ respectively. If there are entries $i',j'$ immediately right of $i,j$ respectively then we must have $i'<j'$ else $j'<i'<j$ would be a Type II co-inversion triple. We are done by induction.
	\end{proof}

	\begin{lemma}\label{lem:flat}
		Suppose row $r$ is strictly longer than row $s$ in an SSKT $S$ with $r<s$. Suppose for some column $c>1$ we have $S(c,r)<S(c,s)$. Then $S(c-1,r)<S(c-1,s)$ as well.
	\end{lemma}
	\begin{proof}
		If not then $S(c,r)<S(c-1,s)<S(c-1,r)$ creates a Type I co-inversion triple.
	\end{proof}
	
	Lascoux \cite{Las03} discovered the Cauchy-like identity
	\begin{equation}\label{eq:cauchy}
		\prod_{i+j\le n+1}(1-x_iy_j)^{-1}=\sum_{\comp{a}\in \mathbb N^n}\atom_{\rev(\comp a)}(x)\key_{\comp a}(y)
	\end{equation}
	which was generalized to classical types by Fu and Lacscoux \cite{FL09}. More recently, Choi and Kwon \cite{CK18} gave a crystal theoretic interpretation and proof. Azenhaus and Emami \cite{AE15} also generalized the result to truncated staircase shapes using Mason's translation of RSK to skyline fillings \cite{Mas09}.
	
	Reversing the alphabet $x_1,\ldots,x_n$ and applying Theorems \ref{thm:key_mac} and \ref{thm:atom_mac}, we get \eqref{eq:cauchy} in the following equivalent form.
	\begin{theorem}[\cite{Las03}]\label{thm:cauchy}
		We have
		\begin{equation}
			\prod_{1\le j\le i\le n}(1-x_iy_j)^{-1}=\sum_{\comp{a}\in \mathbb N^n}E_{\comp a}(x;\infty,\infty)E_{\comp a}(y;0,0).
		\end{equation}
	\end{theorem}

	In Section \ref{sec:rsk} we reprove this fact by describing a bijection, equivalent to RSK, between lower triangular $\mathbb N$ matrices and \[\{(T,S)\in \SSKT\times\rSSAF \mid \sh(S)=\sh(T) \}.\] The section can be skipped without loss of continuity.
	
	The main new result for the current section is now a new interpretation for the coefficients $\tilde K_{\comp a\comp b}$.
	
	\begin{theorem}\label{thm:key-expand}
		The homogeneous complete flagged polynomials expand into key polynomials by
		\[
		\h_{\comp b}=\sum_{\comp a}\tilde K_{\comp a\comp b}\key_{\comp a}
		\]
		where
		\[
		\tilde K_{\comp a\comp b}=|\{ T\in \rSSAF\mid \sh(T)=\comp a,\ \wt(T)=\comp b\}|.
		\]
	\end{theorem}
	\begin{proof}
		We have
		\[
		\prod_{1\le j\le i\le n}(1-x_iy_j)^{-1}=\sum_{\comp{a}\in \mathbb N^n}E_{\comp a}(x;\infty,\infty) E_{\comp a}(y;0,0)=\sum_{\comp{a}\in \mathbb N^n}E_{\comp a}(x;\infty,\infty) \key_{\comp a}(y)
		\]
		using Theorem \ref{thm:cauchy} for the first equality and \ref{thm:key_mac} for the second. Note that the left hand side is
		\[
		\sum_{A \text{ lower triangular }\mathbb{N}-\text{matrix}} x^{\mathrm{row}(A)}y^{\mathrm{col}(A)}
		\]    
		so its $x^{\comp b}$ coefficient is $\h_{\comp b}(y)$ by Lemma \ref{lem:hmatrix}.
		
		We know from Theorem \ref{thm:mac} that
		\[
		E_{\comp a}(x;\infty,\infty)= \sum_{\substack{T\in \rSSAF\\ \sh(T)=\comp a}}x^{\wt(T)}.
		\]
		Then using Theorem \ref{thm:HHL} to take the coefficient of $x^{\comp b}$ in $\sum_{\comp{a}\in \mathbb N^n}E_{\comp a}(x;\infty,\infty) \key_{\comp a}(y)$ we get \[\sum_{\substack{T\in\rSSAF\\\wt(T)=\comp b }} \key_{\sh(T)}(y)=\sum_{\comp a} \tilde K_{\comp a\comp b}\key_{\comp a}\]
		with $\tilde K_{\comp a\comp b}$ as described.
	\end{proof}

	Given this combinatorial interpretation of $\tilde K_{\comp a\comp b}$, Mason \cite{Mas08} showed bijectively that these numbers are related to classical Kostka coefficents by
	\begin{equation}
	K_{\lambda\comp b}=\sum_{\sort(\comp a)=\lambda} \tilde K_{\comp a\comp b}.
	\end{equation}
	This should not be too surprising as one can also consider taking the expansion
	\[
	\h_{\comp b}=\sum_{\comp a}\tilde K_{\comp a\comp b}\key_{\comp a}
	\]
	to its symmetric stable limit, an argument that one can make more precise by using Lemma \ref{lem:iso-subspaces} to see the coefficients don't change when they are suitably identified under the limiting operation. The bijection used by Mason is described explicitly in Section \ref{sec:rsk}.
	
	Since key polynomials are a positive sum of Demazure atoms, Theorem \ref{thm:key-expand} implies the same is true for $\h_{\comp a}$. We can directly find the coefficients using the same tools as before.
	\begin{theorem}
		We have
		\[
		\h_{\comp b}=\sum_{\comp a}\tilde K^{\comp a\comp b}\atom_{\comp a}
		\]
		where
		\[
		\tilde K^{\comp a\comp b}=|\{ T\in \SSKT\mid \sh(T)=\rev(\comp a),\ \wt(T)=\rev(\comp b)\}|.
		\]
	\end{theorem} 
	\begin{proof}
		The left hand side of \eqref{eq:cauchy} is unchanged by interchanging the $x$'s and $y$'s. So
		\[
		\prod_{i+j\le n+1}(1-x_iy_j)^{-1}=\sum_{\comp{a}\in \mathbb N^n}\key_{\comp a}(x)\atom_{\rev(\comp a)}(y).
		\]
		Reversing the alphabet $x_1,\ldots,x_n$ and applying Theorem \ref{thm:key_mac} we get
		\[
		\begin{aligned}
		\prod_{1\le j\le i\le n}(1-x_iy_j)^{-1}&=\sum_{\comp{a}\in \mathbb N^n}\key_{\comp a}(x_n,\ldots,x_1)\atom_{\rev(\comp a)}(y_1,\ldots,y_n)\\
		&=\sum_{\comp{a}\in \mathbb N^n}E_{\comp a}(x_n,\ldots,x_1;0,0)\atom_{\rev(\comp a)}(y_1,\ldots,y_n).
		\end{aligned}
		\]
		The coefficient of $x^{\comp b}$ on the left is $\h_{\comp b}$ as before. Using Theorem \ref{thm:HHL} the coefficient of $x^{\comp b}$ on the right is
		\[
		\sum_{\substack{T\in\SSKT\\\wt(T)=\rev(\comp b)}} \atom_{\rev(\sh(T))}(y_1,\ldots,y_n)=\sum_{\comp a}\tilde K^{\comp a\comp b}\atom_\comp a(y_1,\ldots,y_n)
		\]
		which is therefore equated to $\h_{\comp b}(y_1,\ldots,y_n)$.
	\end{proof}
	Here also there is a relation to the classical Kostka coefficients. Namely
	\begin{equation}
		K_{\lambda\comp b}=\tilde K^{\lambda\comp b}
	\end{equation}
	for a partition $\lambda$. In short, this is because $s_\lambda(x_1,\ldots,x_n)=\key_{\rev(\lambda)}$.
	
	%
	\section{Flagged RSK}\label{sec:rsk}
	%
	The celebrated RSK correspondence associates to each $\mathbb N$-matrix a pair of SSYT of the same shape in such a way that the weights of the SSYT correspond to the column and row sums of the matrix. Myriad resources such as \cite{Sta99} and \cite{Ful97} give an in-depth introduction to the bijection. Mason gave an equivalent characterization of RSK using the combinatorics of semi-skyline augmented fillings \cite{Mas08}, which was generalized by Haglund, Mason and Remmel \cite{HMR13}. Azenhas and Emami \cite{AE15} used Mason's algorithm to prove \eqref{eq:cauchy}, as a special case, which is ultimately our application as well. 
	
	Our RSK analogue is equivalent to the restriction of the full correspondence to lower triangular matrices, and is therefore more specialized than Mason's. In our context it is consequently simpler and more uniquely suited. In comparison to Mason who associates to each matrix a pair of semi-skyline augmented fillings whose shapes are rearrangements of each other, we put the lower triangular matrices into explicit bijection with pairs of SSKT and reverse SSAF that share a shape. Just as the RSK algorithm shows how the monomial expansion of a complete symmetric function factors through the Schur expansion, our ``flagged'' RSK specialization is shows how the monomial expansion of $\h_{\comp a}$ factors through the key expansion.
	
	Up until this point there has been an implicit dependence on a fixed value $n$, particularly in the definitions of our combinatorial objects $\SSYT,\SSKT,\rSSAF$. In this section it will help to make the dependence explicit, so we instead use the notation $\SSYT_n,\SSKT_n,\rSSAF_n$.
	
	There is one more combinatorial object we must introduce in order to continue.
	\begin{definition}
		A \textbf{reverse semi-standard Young tableau} (reverse SSYT) is a filling of a Ferrers diagram in the French convention with entries such that columns strictly decrease bottom to top and rows weakly decrease left to right. The set of reverse SSYT with entries in $[n]$ is denoted $\rSSYT_n$.
	\end{definition}

	For any set $\SSYT_n,\rSSYT_n,\rSSAF_n,\SSKT_n$ of tableau-like objects we may append an argument restricting the shape. For instance
	\[
	\SSKT_n(\comp a)=\{ S\in\SSKT_n\mid \sh(T)=\comp a \}.
	\]
	
	First we review the full RSK correspondence. Given a reverse SSYT $P$ and positive integer $j$, we define $P\leftarrow j$ to be the reverse SSYT obtained by the following procedure.
	\begin{enumerate}
		\item Let $r=1$.
		\item Place $j$ in the leftmost position in the $r$th row from the bottom not occupied by a weakly larger entry, removing the entry $j'$ that occupies the position if necessary.
		\item If an entry $j'$ was indeed removed from the row, go back to step 2 replacing the values $j$ with $j'$, and $r$ with $r+1$.
	\end{enumerate}

	This insertion algorithm is the central component of RSK. We also need to know how to interpret an $\mathbb N$-matrix in $\mathcal M_n$ as a \newword{biword}, which is a multiset in the alphabet $\left\{ \begin{pmatrix} i\\ j \end{pmatrix} \mid i,j\in[n] \right\}$. A matrix $A$ corresponds to the biword whose number of elements $\begin{pmatrix} i\\ j \end{pmatrix}$ is equal to the $i,j$ entry of $A$. We represent a biword as a two-line array $\begin{pmatrix} i_1&i_2&\cdots&i_\ell\\ j_1&j_2&\cdots&j_\ell \end{pmatrix}$ where $\begin{pmatrix} i_k\\ j_k \end{pmatrix}$ are the elements of the biword, $i_k\le i_{k+1}$, and if $i_k=i_{k+1}$ then $j_k\ge j_{k+1}$.
	
	Let
	\begin{equation}
		\pairs_n=\{ (P,Q)\in\rSSYT_n\times \SSYT_n \mid \sh(P)=\sh(Q)  \}.
	\end{equation}
	Given $(P,Q)\in \pairs_n$ define
	\[
	(P,Q)\leftarrow \begin{pmatrix} i\\ j \end{pmatrix} = (P\leftarrow j, Q')
	\]
	where $Q'$ is obtained from $Q$ by adding the box $\sh(P\leftarrow j)\setminus\sh(P)$ with entry $i$.
	
	Then we define 
	\[
	A \overset{\RSK}{\mapsto} \left(\cdots\left(\left((\emptyset,\emptyset)\leftarrow \begin{pmatrix} i_1\\ j_1 \end{pmatrix} \right)\leftarrow \begin{pmatrix} i_2\\ j_2 \end{pmatrix} \right)\cdots\leftarrow \begin{pmatrix} i_\ell\\ j_\ell \end{pmatrix}\right)
	\]
	where $A$ is the matrix corresponding to the biword $\begin{pmatrix} i_1&i_2&\cdots&i_\ell\\ j_1&j_2&\cdots&j_\ell \end{pmatrix}$. This is a bijection from $\mathcal M_n$ to $\pairs_n$ which differs from the usual construction only in the superficial way that we have replaced $\SSYT\times\SSYT$ with $\rSSYT\times\SSYT$. When $A \overset{\RSK}{\mapsto} (P,Q)$ we call $P$ the \newword{insertion tableau} and $Q$ the \newword{recording tableau}. By construction, $\wt(P)$ and $\wt(Q)$ are the column and row sums respectively of the corresponding matrix.

	As an example, the tableaux in Figure \ref{fig:RSK} are the image of
	\[
	\setcounter{MaxMatrixCols}{13}
	\begin{pmatrix} 1&3&3&4&4&4&5&5&5&5&6&7&7\\1&3&2&4&3&1&4&4&3&2&1&6&3  \end{pmatrix}.
	\]
	The final insertion computation $\leftarrow 3$ is shown in Figure \ref{fig:classical-insertion}.
	\begin{figure}[ht]
		\begin{ytableau}
			1\\
			2\\
			3&1\\
			4&3&2\\
			6&4&3&3&3&1
		\end{ytableau}\hspace{100pt}
		\begin{ytableau}
		7\\
		5\\
		4&7\\
		3&4&5\\
		1&3&4&5&5&6
	\end{ytableau}
	\caption{\label{fig:RSK} A pair of fillings in $\pairs_n$.}
	\end{figure}

	\begin{figure}[ht]
		\begin{tikzpicture}
			\node at (0,0) (A) {
				\begin{ytableau}
					1\\
					2\\
					3\\
					4&3&1\\
					6&4&3&3&2&1&\none[\ \leftarrow]&\none[3]
				\end{ytableau}
			};
			
			\node at (3.2,0) (A) {
				\begin{ytableau}
					1\\
					2\\
					3\\
					4&3&1&\none&\none&\none&\none[\ \leftarrow]&\none[2]\\
					6&4&3&3&*(green)3&1
				\end{ytableau}
			};
			\node at (6.4,0) (A) {
				\begin{ytableau}
					1\\
					2\\
					3&\none&\none&\none&\none&\none&\none[\ \leftarrow]&\none[1]\\
					4&3&*(green)2\\
					6&4&3&3&3&1
				\end{ytableau}
			};
			\node at (9.6,0) (A) {
				\begin{ytableau}
					1\\
					2\\
					3&*(green)1\\
					4&3&2\\
					6&4&3&3&3&1
				\end{ytableau}
			};
		\end{tikzpicture}
		\caption{\label{fig:classical-insertion}The insertion procedure $\leftarrow3$.}
	\end{figure}
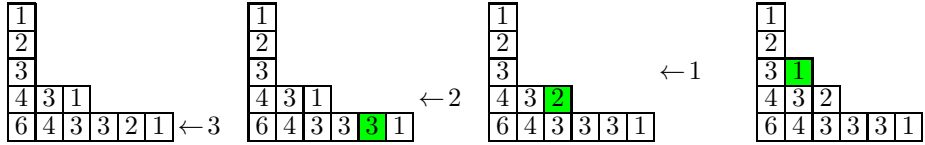
	
	Before describing the flagged RSK algorithm, let us see how its proposed image \begin{equation}
		\fpairs_n=\{ (S,T)\in \SSKT_n\times \rSSAF_n \mid \sh(S)=\sh(T) \}
	\end{equation}
	embeds into $\pairs_n$.
	
	\begin{definition}[\cite{Mas08}]
		The map
		\[
		\rho:\rSSAF_n\to\SSYT_n
		\]
		takes $T\in\rSSAF_n$ and yields the unique SSYT $\rho(T)$ whose $i$th column consists of the same set of entries as does the $i$th column of $T$.
	\end{definition}

	The map is clearly weight-preserving. It is also a well-defined bijection where $\rho^{-1}(Q)$ is constructed column by column from left to right by placing each entry in the corresponding column of $Q$, from smallest to largest, into the topmost available position immediately right of a weakly lesser entry (allowing the basement column into consideration) \cite{Mas08}. See Figure \ref{fig:embed}.

	\begin{definition}[\cite{ASch18}]
		The map
		\[
		\tau:\SSKT_n\to\rSSYT_n
		\]
		takes $S\in \SSKT_n$ and yields the unique reverse SSYT $\tau(S)$ whose $i$th column consists of the same set of entries as does the $i$th column of $S$.
		
		Let 
		\[
		\tau_{\comp a}:\SSKT_n(\comp a)\to\rSSYT_n(\sort(\comp a))
		\] denote the restriction of $\tau$.
	\end{definition}
	In contrast to $\rho$, $\tau$ is not a bijection, but $\tau_{\comp a}$ is a well-defined embedding. The latter has the left inverse $\tau_{\comp a}^\dagger$ which takes $P\in \rSSYT_n(\sort(\comp a))$ and yields a filling $\tau^\dagger_{\comp a}(P):\D(\comp a)\to [n]$ defined column by column from right to left and bottom to top, at each cell selecting the smallest remaining entry in the column set that maintains the decreasing row condition \cite{ASch18}. The filling $\tau^\dagger_{\comp a}(P)$ is an SSKT exactly when $P$ is in the image of $\tau_{\comp a}$.  Once again, see Figure \ref{fig:embed}.
	
	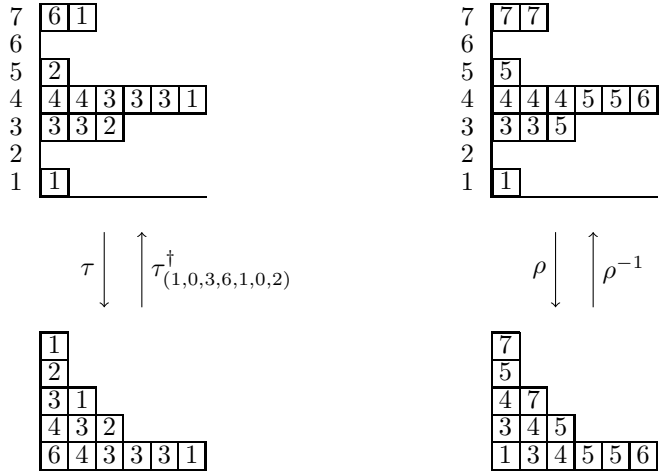
\begin{figure}[ht]
		\begin{tikzpicture}
			\node at (0,0) (A) {
			\begin{ytableau}
				\none[7]\\\none[6]\\\none[5]\\\none[4]\\\none[3]\\\none[2]\\\none[1]\\
			\end{ytableau}
			\vline
			\begin{ytableau}
				6&1\\
				\none \\
				2\\
				4&4&3&3&3&1\\
				3&3&2\\
				\none\\
				1\\
				\hline
			\end{ytableau}
			};
		
			\node at (6,0) (B) {
			\begin{ytableau}
				\none[7]\\\none[6]\\\none[5]\\\none[4]\\\none[3]\\\none[2]\\\none[1]\\
			\end{ytableau}
			\vline
			\begin{ytableau}
				7&7\\
				\none \\
				5\\
				4&4&4&5&5&6\\
				3&3&5\\
				\none\\
				1\\
				\hline
			\end{ytableau}
			};
		
			\node at (0.25,-4) (C) {
				\begin{ytableau}
					1\\
					2\\
					3&1\\
					4&3&2\\
					6&4&3&3&3&1
				\end{ytableau}
			};
		
			\node at (6.25,-4) (D) {
				\begin{ytableau}
					7\\
					5\\
					4&7\\
					3&4&5\\
					1&3&4&5&5&6
				\end{ytableau}
			};
		
			\draw [-to] (0,-1.75)--(0,-2.75) node[midway,left]{$\tau$};
			\draw [-to] (0.5,-2.75)--(0.5,-1.75) node[midway,right]{$\tau^\dagger_{(1,0,3,6,1,0,2)}$};
			\draw [-to] (6,-1.75)--(6,-2.75) node[midway,left]{$\rho$};
			\draw [-to] (6.5,-2.75)--(6.5,-1.75) node[midway,right]{$\rho^{-1}$};
		\end{tikzpicture}
		\caption{\label{fig:embed}An illustration of the maps $\tau$ and $\rho$.}
	\end{figure}

	Abusing notation, we let $\tau\times\rho$ refer to the actual product map restricted to $\fpairs_n\to\pairs_n$. This is our embedding.

	\begin{lemma}\label{lem:subset}
		The map 
		\[
		\tau\times\rho:\fpairs_n\to\pairs_n
		\]
		is well-defined and injective with the left inverse
		\[
		(P,Q)\mapsto (\tau^\dagger_{\sh(\rho^{-1}(Q))}(P), \rho^{-1}(Q)).
		\]
		.
	\end{lemma}
	\begin{proof}
		Let $S\in\SSKT_n$. The partition conjugate to $\sort(\sh(S))$ is also the partition conjugate to $\sh(\tau(S))$, the $i$th part of either being the number of cells in column $i$ of $S$. Thus $\sort(\sh(S))=\sh(\tau(S))$. In the same way, $\sort(\sh(T))=\sh(\rho(T))$ for $T\in\rSSAF_n$. Then for $(S,T)\in\fpairs_n$ we have $\sh(S)=\sh(T)$ which implies $\sh(\tau(S))=\sh(\rho(T))$ so indeed $(\tau\times\rho)(S,T)\in \pairs_n$ and the map is well-defined.
		
		For $(S,T)\in\fpairs_n$ let $(P,Q)=(\tau(S),\rho(T))$. We have $\sh(S)=\sh(T)$ so
		\begin{align*}
		(\tau^\dagger_{\sh(\rho^{-1}(Q))}(P), \rho^{-1}(Q))&=(\tau^\dagger_{\sh(\rho^{-1}(\rho(T)))}(\tau(S)), \rho^{-1}(\rho(T)))\\&=(\tau^\dagger_{\sh(T)}(\tau_{\sh(T)}(S)),T)\\&=(S,T)
		\end{align*}
		which demonstrates the left inverse for $\tau\times\rho$.
	\end{proof}
	\begin{remark}
		Since SSYT and reverse SSYT are completely determined by their column sets, Lemma \ref{lem:subset} says that an element of $\fpairs_n$ is also completely determined by the column sets of its constituent fillings.
	\end{remark}

	We now construct the flagged RSK correspondence. Given $S\in\SSKT_n$ and $j\in[n]$ define $S\leftharpoondown_n j$ to be the filling obtained as follows, always allowing entries in the basement column into consideration.
	\begin{enumerate}
		\item Let $c=\infty$.
		\item Let $c'\le c$ be maximal such that column $c'$ contains strictly fewer entries weakly greater than $j$ than does column $c'-1$.
		\item Place the entry $j$ in the topmost position in column $c'$ not occupied by a weakly larger entry, and immediately right of an entry weakly greater than $j$. Remove the entry $j'$ that occupies the position if it exists.
		\item If an entry $j'$ was indeed removed, go back to step 2 replacing the values $c$ with $c'$ and $j$ with $j'$.
	\end{enumerate}

	\begin{remark}
		We consider $\SSYT_n$, $\rSSYT_n$, $\rSSAF_n$, $\SSKT_n$, $\pairs_n$, $\fpairs_n$ to all be subsets of $\SSYT_{n+1}$, $\rSSYT_{n+1}$, $\rSSAF_{n+1}$, $\SSKT_{n+1}$, $\pairs_{n+1}$, $\fpairs_{n+1}$ respectively in the obvious way. 
		
		In particular, we may write $S\leftharpoondown_{n} j$ for $S\in\SSKT_i$ with $i\le n$. The subtlety obscured by the notation $\hat S$ is that the number of cells in the basement depends on whether we view $S$ as an element of $\SSKT_n$ or $\SSKT_i$. When we write $S\leftharpoondown_{n} j$ we are implicitly considering $S$ as an element of $\SSKT_n$, and because of the basement this is not necessarily equivalent to $S\leftharpoondown_{i} j$.
	\end{remark}

	Given $(S,T)\in \fpairs_n$ and $j\in[n]$ define
	\[
	(S,T)\leftharpoondown \begin{pmatrix} n\\ j \end{pmatrix} = (S\leftharpoondown_n j, T')
	\]
	where $T'$ is obtained from $T$ by adding the cell $\D(\sh(S\leftharpoondown_n j))\setminus\D(\sh(S))$ with entry $n$.
	
	Given a biword $\begin{pmatrix} i_1&i_2&\cdots&i_\ell\\ j_1&j_2&\cdots&j_\ell \end{pmatrix}$ with each $j_k\le i_k$, i.e. corresponding to a lower triangular matrix $L$, we define
	\[
	L \overset{\fRSK}{\mapsto} \left(\cdots\left(\left((\emptyset,\emptyset)\leftharpoondown \begin{pmatrix} i_1\\ j_1 \end{pmatrix} \right)\leftharpoondown \begin{pmatrix} i_2\\ j_2 \end{pmatrix} \right)\cdots\leftharpoondown \begin{pmatrix} i_\ell\\ j_\ell \end{pmatrix}\right).
	\]
	
	The pair of fillings in Figure \ref{fig:fillings} is in fact the image of 
	\[
	\setcounter{MaxMatrixCols}{13}
	\begin{pmatrix} 1&3&3&4&4&4&5&5&5&5&6&7&7\\1&3&2&4&3&1&4&4&3&2&1&6&3  \end{pmatrix}
	\]
	whose final insertion computation $\leftharpoondown_7 3$ is shown in Figure \ref{fig:insertion}.
	
	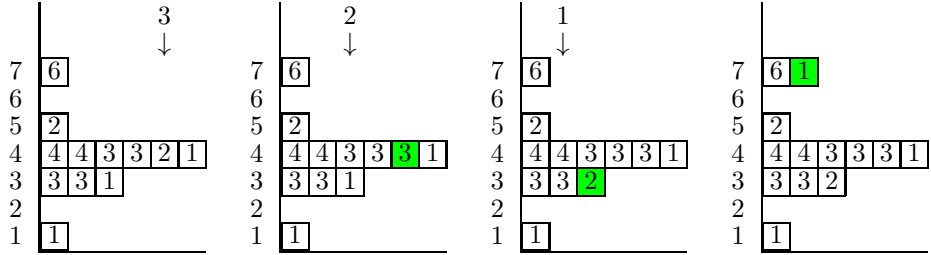
\begin{figure}[ht]
		\begin{tikzpicture}
			\node at (0,0) (A) {
				\begin{ytableau}
					\none\\\none\\\none[7]\\\none[6]\\\none[5]\\\none[4]\\\none[3]\\\none[2]\\\none[1]\\
				\end{ytableau}
				\vline
				\begin{ytableau}
					\none&\none&\none&\none&\none[3]&\none\\
					\none&\none&\none&\none&\none[\downarrow]&\none\\
					6\\
					\none \\
					2\\
					4&4&3&3&2&1\\
					3&3&1\\
					\none\\
					1\\
					\hline
				\end{ytableau}
			};
		
			\node at (3.2,0) (A) {
				\begin{ytableau}
					\none\\\none\\\none[7]\\\none[6]\\\none[5]\\\none[4]\\\none[3]\\\none[2]\\\none[1]\\
				\end{ytableau}
				\vline
				\begin{ytableau}
					\none&\none&\none[2]&\none&\none&\none\\
					\none&\none&\none[\downarrow]&\none&\none&\none\\
					6\\
					\none \\
					2\\
					4&4&3&3&*(green) 3&1\\
					3&3&1\\
					\none\\
					1\\
					\hline
				\end{ytableau}
			};
			\node at (6.4,0) (A) {
				\begin{ytableau}
						\none\\\none\\\none[7]\\\none[6]\\\none[5]\\\none[4]\\\none[3]\\\none[2]\\\none[1]\\
				\end{ytableau}
				\vline
				\begin{ytableau}
					\none&\none[1]&\none&\none&\none&\none\\
					\none&\none[\downarrow]&\none&\none&\none&\none\\
					6\\
					\none \\
					2\\
					4&4&3&3&3&1\\
					3&3&*(green)2\\
					\none\\
					1\\
					\hline
				\end{ytableau}
			};
			\node at (9.6,0) (A) {
				\begin{ytableau}
					\none\\\none\\\none[7]\\\none[6]\\\none[5]\\\none[4]\\\none[3]\\\none[2]\\\none[1]\\
				\end{ytableau}
				\vline
				\begin{ytableau}
					\none&\none&\none&\none&\none&\none\\
					\none&\none&\none&\none&\none&\none\\
					6&*(green)1\\
					\none \\
					2\\
					4&4&3&3&3&1\\
					3&3&2\\
					\none\\
					1\\
					\hline
				\end{ytableau}
			};
		\end{tikzpicture}
		\caption{\label{fig:insertion}The insertion procedure $\leftharpoondown_73$.}
	\end{figure}
	
	\begin{lemma}\label{lem:insertion}
		If $S\in\SSKT_n$, and $j_0\in[n]$, then the operation $S\leftharpoondown_{n}j_0 $ is well-defined and the result is in $\SSKT_n$.
	\end{lemma}
	\begin{proof}
		Suppose we are at the beginning of step (2) of the insertion $S\leftharpoondown_{n} j_0$ and we have reached step (4) exactly $t$ times. Let $j_0,j_1,\ldots j_t$ be the successive values of $j$ in the procedure so far, $\infty=c_0,c_1,\ldots, c_t$ the successive values of $c$, and $S=S_0,S_1,\ldots, S_t$ the successive intermediate fillings. Assume that the procedure thus far has been well-defined, and $S_t\in\SSKT_n$. It can be seen from the insertion definition that $j_0>j_1>\cdots>j_t$.
		
		Our first order of business is to show that the next steps (2), (3), (4) of the insertion are well-defined. If $c'$ is determined as in step (2) then step (3) is well-defined, and (4) certainly is as well. So we need only show step (2) makes sense.
		
		If $t=0$ so that $c=\infty$, a column $c'$ as described in step (2) certainly exists as the basement column contains the entry $j_0$.
		
		Suppose $t>0$. Column $c_t$ cannot contain an entry equal to $j_t$, since an entry $j_t$ has just been replaced by a strictly larger entry and the previous filling could only have had a single entry $j_t$ in the column due to the non-attacking condition on $S_t$. Then there is a leftmost column $b\le c$ which does not contain the entry $j_t$ in row $j_t$, and $b>0$. Since the rows of $S_t$ weakly decrease left to right and there is an entry $j_t$ in row $j_t$ column $b-1$, any entry in row $j_t$ column $b$ must be strictly less than $j_t$. Therefore column $b$ contains strictly fewer entries weakly greater than $j_t$ than does column $b-1$, so step (2) is well-defined.
		
		Now we must show that the filling $S_{t+1}$ obtained by applying step (3) is an SSKT.
		
		\begin{description}
			\item[$\maj(\hat S_{t+1})=0$] The entry immediately left of the insertion position is weakly greater than $j_t$, and if an entry exists immediately to the right of the insertion position then it is weakly less than the entry $j_{t+1}<j_t$ that $j_t$ replaces. Therefore $\maj(\hat S_{t+1})$ remains zero.
			
			\item[$\hat S_{t+1}$ is non-attacking] We have seen above that column $c_t$ of $S_t$ does not contain an entry $j_t$. Then none of the columns $c',c'+1,\ldots, c_t$ can contain an entry $j_t$ without contradicting the choice of $c'$. Therefore no two cells share an entry in column $c'$ of $S_{t+1}$.
			
			If there is an entry $j_t$ in column $c'-1$ strictly above the row of the newly inserted entry $j_t$, then the position immediately right of the former entry cannot contain another entry $j_t$, but this contradicts the insertion position. Therefore $\hat S_{t+1}$ remains non-attacking.
			
			\item[$\hat S_{t+1}$ contains no Type I co-inversion triples] Because $\maj(\hat S_{t+1})=0$, recall that the only possible Type I co-inversion triples are of the form $u<v<w$ with $u,w$ in the strictly longer row $r$ below the row $s$ containing $v$ (cf. $i<j<k$ in Fig. \ref{fig:inv}).
			
			If $j_t$ takes the role of $u$, then there must be an entry $v'>j_t$ immediately right of $v$ to justify the insertion position. Then there must be an entry immediately right of $j_t$ since the row is strictly longer than row $s$. This means that in $\hat S_t$ there must have been an entry $j_{t+1}<j_t$ occupying the position of $j_t$. Then $j_{t+1}<v<w$ is a co-inversion triple in $\hat S_t$ which is a contradiction.
			
			If $j_t$ takes the role of $w$, then we must have that the previous entry $j_{t+1}<j_t$ in the position satisfies $j_{t+1}< v$. To be consistent with the choice of insertion position, the entry $v'$ immediately left of $v$ satisfies $v'<j_t$. The entry $w'$ immediately left of the insertion position satisfies $w'\ge j_{t}>v'\ge v>j_{t+1}$. In particular $j_{t+1}<v'<w'$ is a co-inversion triple in $\hat S_t$, a contradiction.
			
			Suppose $j_t$ takes the role of $v$. Since $u<j_t<w$ with $u$ in column $c'+1$, we must have $c_t=c'$ and $t>0$ by choice of $c'$. Now we know that $j_{t}$ lies in column $c_t=c'$ of $\hat S_{t-1}$. It must lie weakly below row $r$, else $j_t<w$ contradicts Lemma \ref{lem:ledge}, or we again have a co-inversion triple $u<j_t<w$. Say that $v'$ is the entry in $(c_t-1,s)$, i.e. immediately left of $j_t$ in $\hat S_{t+1}$ and $j_{t+1}$ in $\hat S_{t-1}$. In $\hat S_{t-1}$ the row containing $j_t$ strictly below row $s$ must be strictly longer than row $s$, else $j_{t+1}<j_t<v'$ is a Type II co-inversion triple. Since $j_{t-1}$ is inserted into the position occupied by $j_t$ in $\hat S_{t-1}$, and $j_{t-1}>j_{t+1}$, we must have $v'<j_{t-1}\le p$ to justify this insertion position. Then $j_t<v'<p$ is a co-inversion triple in $\hat S_{t-1}$ which is impossible.
			
			Finally, we must consider that in obtaining $\hat S_{t+1}$ we may have changed the relative length of rows and to create new Type I triples that do not contain the newly inserted entry. That is, suppose row $r,s$ are the same length in $\hat S_{t}$, but row $r$ is strictly longer in $\hat S_{t+1}$. By Lemma \ref{lem:ledge} every entry in row $s$ is strictly greater than the entry in the same column of row $r$ in $\hat S_{t}$. This carries over to $\hat S_{t+1}$ which means none of the new Type I triples can be co-inversion triples.
			
			\item [$\hat S_{t+1}$ contains no Type II co-inversion triples]
			Because $\maj(\hat S_{t+1})=0$, the only Type II co-inversion triples $u<v<w$ have $u,w$ in the weakly longer row $s$ above the row $r$ containing $v$. 
			
			First we consider the special case where we have such a co-inversion triple in a pair of rows whose relative lengths have been changed by the insertion of $j_t$. This is to say, in $\hat S_{t+1}$ the entry $j_t$ lies in the insertion position $(c',s)$, and some entry $q$ of the cell $(c',r)$ is the rightmost entry in row $r$ of both $\hat S_t$ and $\hat S_{t+1}$. Let $q'$ be the entry immediately left of $q$, and $p$ the entry of $(c'-1,s)$. If $q<j_t$ then $q<p$ which means $p>q'$ to avoid a Type I co-inversion triple in $\hat S_t$. In this case, by Lemma \ref{lem:flat}, every entry of row $s$ is strictly greater than the entry in row $r$ of the same row, in both $\hat S_t$ and $\hat S_{t+1}$. This prevents any co-inversion triples in these rows.
			
			If instead $q>j_t$, recalling the position immediately right of $q$ is empty, by choice of insertion column $c'$ we must have that $c'=c_t$ and $t>0$. Then in $\hat S_{t-1}$ we must have the entry $j_t$ in column $c_t$, necessarily weakly below row $r$ by Lemma \ref{lem:ledge}. Let $h$ be the entry immediately left of $j_t$ in $\hat S_{t-1}$. We must have $p<h$ since $j_{t-1}$ gets inserted immediately right of $h$ rather than $p$. This leads to the Type I co-inversion triple $j_t<p<h$ in $\hat S_{t-1}$ which is a contradiction.
			
			Now we consider the Type II co-inversion triples that contain the inserted entry $j_t$. If $j_t$ takes the role of $u$, the same position must be empty in $\hat S_t$ else it would be occupied by some $j_{t+1}<j_t$ leading to the co-inversion triple $j_{t+1}<v<w$. This puts us back into the above special case which has already been discounted.
			
			We cannot have $j_t$ take the role of $v$ as this would contradict the choice of insertion position.
			
			If $j_t$ takes the role of $w$ then it must have replaced an entry $j_{t+1}$ in $\hat S_{t}$. So the row $s$ is weakly longer than row $r$ in $\hat S_t$ as well. By Lemma \ref{lem:ledge} we then have $u>v$, a contradiction.
		\end{description}
		We are done by induction on $t$.
	\end{proof}

	We now see that the flagged insertion algorithm is a special case of the usual insertion algorithm for RSK.

	\begin{proposition}\label{prop:insertion}
		For $j_0\in[n]$ the diagram
		\begin{center}
			\begin{tikzcd}
				\SSKT_n \arrow[r, "\leftharpoondown_n j_0"] \arrow[d, "\tau"]
				& \SSKT_n \arrow[d, "\tau"] \\
				\rSSYT_n \arrow[r, "\leftarrow j_0"]
				& \rSSYT_n
			\end{tikzcd}
		\end{center}
		commutes.
	\end{proposition}
	\begin{proof}
		Consider step (2) of the insertion algorithm $\tau(S)\leftarrow j_0$. We have some value of $j$ we must insert, some row value $r$, and also some column $c$ from which the entry $j$ was just removed (taking $c=\infty$ if $j=j_0$). The entry $j$ is to be inserted in the leftmost position in the $r$th row not occupied by a weakly larger entry, say in column $c'$. Before applying step (2) of the algorithm, column $c'$ then contains exactly $r-1$ entries weakly greater than $j$ since the column entries are sorted. If $c'>1$ then the entry in row $r$ column $c'-1$ is weakly greater than $j$ by choice of insertion position. Then if $c'>1$, column $c'$ contains strictly fewer entries strictly greater than $j$ than does column $c'-1$. In fact $c'$ is the maximal column index $c'\le c$ with this property as it is exactly the lowest $r-1$ entries in columns $c'+1,c'+2,\ldots, c$ that are weakly greater than $j$.
		
		This is to say we can restate the algorithm $\tau(S)\leftarrow j_0$ in terms of column sets as follows.
		
		\begin{enumerate}
			\item Set $j=j_0$ and $c=\infty$.
			\item Let $c'$ be the rightmost column $c'\le c$ such that either $c'=1$ or column $c'$ contains strictly fewer entries strictly greater than $j$ than does column $c'-1$. Let $j'$ be the largest entry $j'<j$ such that columns $c'-1$ and $c'$ contain the same number of entries weakly greater than $j'$, if it exists.
			
			\item Place $j$ in column $c'$ and remove $j'$ if it exists.
			\item Go back to step (2) setting $j=j'$ and $c=c'$.
		\end{enumerate}
		
		Note that if $j'$ did not satisfy the condition we place on it, then it would be reinserted into the same column, a step that this characterization freely ignores. If we show that the column sets of $S\leftharpoondown_n j_0$ are characterized in this same way then we are done by definition of $\tau$.
		
		Now considering step (3) of the insertion $S\leftharpoondown_n j_0$, if there are strictly fewer entries weakly greater than $j'$ in column $c'$ than in column $c'-1$, then this is still true after replacing $j'$ by $j$. Then $j'$ will be inserted into the same column $c'$, which is ignorable if we care only about column sets.
		
		Suppose there is an entry $k$ in column $c'$, maximal such that $k<j$ and both columns $c'$ and $c'-1$ contain the same number of entries weakly greater than $k$. If $j$ is to be inserted immediately right of the entry $u\ge j>k$ then by choice of $k$ there exists $j'$ immediately right of $u$ with $j'\ge k$. If $j'>k$ then $j'$ will be inserted back into the same columns as above.
		
		We conclude as follows that the column set characterization above holds for $S\leftharpoondown_n j_0$. If there is no entry $k$ as described then we may terminate they algorithm as any additional steps will see us merely adding and removing entries from the same fixed column. If there is such an entry $k$ then the algorithm for $S\leftharpoondown_n j_0$ will eventually remove $k$ from the column, and any intermediate steps are ignorable in terms of column sets. The entry $k$ is determined exactly the same way as $j'$ in the column set characterization for $\tau(S)\leftarrow j_0$, so we have found that the same algorithm computes the column sets of $S\leftharpoondown_n j_0$.
	\end{proof}
	\begin{remark}
		Notice that the operation
		\[
		\left(\begin{ytableau}
			\none[3]\\\none[2]\\\none[1]\\
		\end{ytableau}
		\vline
		\begin{ytableau}
			\none\\
			2\\
			1\\
			\hline
		\end{ytableau} \quad\leftharpoondown_33\right)=
		\begin{ytableau}
		\none[3]\\\none[2]\\\none[1]\\
		\end{ytableau}
		\vline
		\begin{ytableau}
		3\\
		2\\
		1\\
		\hline
		\end{ytableau}
		\]
		requires only one iteration of its defining algorithm whereas the equivalent insertion
		\[
		\left(\begin{ytableau}
			1\\
			2\\
		\end{ytableau} \quad\leftarrow3\right)=
		\begin{ytableau}
			1\\
			2\\
			3\\
		\end{ytableau}
		\]
		requires more replacements to be made: 2 by 3 and 1 by 2. In fact it is true in general that the classical insertion algorithm requires at least as many applications of its core loop as does the equivalent flagged operation. Consider that when an entry $j$ is added to a column set by classical insertion, each existing entry in the column less than $j$ and greater than the entry $j'$ that is ultimately removed from the column (if it exists) must have its position shifted up. Meanwhile the flagged insertion algorithm shuffles some subset of these entries.
	\end{remark}

	We are ready to show the flagged RSK algorithm is a restriction of the classical algorithm.

	\begin{theorem}\label{thm:rsk}
		The map $\fRSK:\mathcal L_n\to \fpairs_n$ is well-defined and the diagram
		\begin{center}
			\begin{tikzcd}
				\mathcal L_n \arrow[r, "\fRSK"] \arrow[d,hookrightarrow]
				& \fpairs_n \arrow[d,hookrightarrow, "\tau\times\rho"] \\
				\mathcal M_n \arrow[r, "\RSK"]
				& \pairs_n
			\end{tikzcd}
		\end{center}
		commutes.
	\end{theorem}
	\begin{proof}
		Assume that for any biword of length at most $\ell$, its image under the flagged RSK algorithm is in $\fpairs_n$ and the maps commute. Take a biword $\begin{pmatrix} i_1&\cdots&i_{\ell+1}\\ j_1&\cdots&j_{\ell+1} \end{pmatrix}$ with each $j_k\le i_k$. Let $(S,T)$ and $(P,Q)$ be the respective images of $\begin{pmatrix} i_1&\cdots&i_{\ell}\\ j_1&\cdots&j_{\ell} \end{pmatrix}$ under $\fRSK$ and $\RSK$, so that $(P,Q)=(\tau(S),\rho(T))$ by assumption. Let
		\[
		(S',T')=\fRSK \begin{pmatrix} i_1&\cdots&i_{\ell+1}\\ j_1&\cdots&j_{\ell+1} \end{pmatrix} = (S,T)\leftharpoondown \begin{pmatrix} i_{\ell+1}\\ j_{\ell+1} \end{pmatrix}
		\]
		and
		\[
		(P',Q')=\RSK \begin{pmatrix} i_1&\cdots&i_{\ell+1}\\ j_1&\cdots&j_{\ell+1} \end{pmatrix} = (P,Q)\leftarrow  \begin{pmatrix} i_{\ell+1}\\ j_{\ell+1} \end{pmatrix}.
		\]
		
		By Proposition \ref{prop:insertion} we know $\tau(S')=P'$. Then, the column $c$ of $\D(\sh(S'))\setminus \D(\sh(S))$ is also the column of $\sh(P')\setminus \sh(P)$. Therefore both $T'$ and $Q'$ are obtained by adding the entry $i_{\ell+1}$ to column $c$ of $T$ and $Q$ respectively. Since $Q=\tau(T)$ the column sets of $T'$ and $Q'$ still match and the only question is whether $T'$ is in $\rSSAF_n$. We will show $\rho^{-1}(Q')=T'$ to prove this is the case.
		
		Note that for a reverse SSYT $R$ and $p\ge q$, the box added from $R\leftarrow p$ is in a column strictly left of the second box added from $(R\leftarrow p)\leftarrow q$. This is because we can inductively assume the first entry inserted into a row $r$ is weakly greater than the second entry inserted into row $r$ if it exists, and consequently the first entry removed from the row is strictly left of and weakly greater than the second removed entry, if it exists. In particular, this fact implies any entry of $Q'$ strictly right of column $c$ is strictly less than $i_{\ell+1}$.
		
		The constructions of $\rho^{-1}(Q')$ and $\rho^{-1}(Q)$ are identical in the first $c-1$ columns where the column sets are identical. In the construction of $\rho^{-1}(Q')$, the entry $i_{\ell+1}$ is the last to be added to column $c$, so the positions of all other entries of the column coincide in both reverse SSAF. Since all entries strictly right of column $c$ are strictly less than $i_{\ell+1}$, the rest of the $\rho^{-1}(Q')$ and $\rho^{-1}(Q)$ construction proceed identically. Therefore, $\rho^{-1}(Q')$ can be obtained from $T=\rho^{-1}(Q)$ by placing the entry $i_{\ell+1}$ in the topmost empty position of $\hat T$ immediately right of an occupied cell (treating $T$ as an element of $\rSSAF_{i_{\ell+1}}$).
		
		By Lemma \ref{lem:ledge} the entries of column $c-1$ of $\hat S$ immediately left of an empty position must be increasing from bottom to top. It follows that the cell $\D(\sh(S'))\setminus \D(\sh(S))$ is also the topmost empty position of $\hat S$ immediately right of an occupied cell (treating $S$ and an element of $\rSSAF_{i_{\ell+1}}$). Therefore $\rho^{-1}(Q')=T'$ which shows $T'\in\rSSAF_n$. We are done by induction on $\ell$.
	\end{proof}

	We can now show that the flagged RSK algorithm is indeed a bijection as with the classical algorithm.
	\begin{theorem}\label{thm:bijection}
		The map $\fRSK:\mathcal L_n\to \fpairs_n$ is a bijection. Moreover, for $L\in \mathcal L_n$ and $(S,T)=\fRSK(L)$, the column and row sums of $L$ respectively correspond to $\wt(S)$ and $\wt(T)$.
	\end{theorem}
	\begin{proof}
		The second assertion is by construction. Since $\RSK:\mathcal M_n\to \fpairs_n$ is a bijection, it is also immediate from Theorem \ref{thm:rsk} that $\fRSK$ is injective.
		
		Let $F: \mathcal M_n\to\mathcal L_{2n}$ be the map
		\[
		\begin{pmatrix} i_1&\cdots&i_{\ell}\\ j_1&\cdots&j_{\ell} \end{pmatrix}\mapsto\begin{pmatrix} i_1+n&\cdots&i_{\ell}+n\\ j_1&\cdots&j_{\ell} \end{pmatrix}
		\]
		to see we have the following commutative diagram
		\begin{center}
			\begin{tikzcd}[sep=large]
				\mathcal L_n \arrow[d, "\fRSK"] \arrow[r,hookrightarrow]
				& \mathcal M_n \arrow[d, "\RSK"]  \arrow[r,hookrightarrow, "F"]
				& \mathcal L_{2n} \arrow[d, "\fRSK"] \arrow[r,hookrightarrow]
				& \mathcal M_{2n} \arrow[d, "\RSK"] \\
				\fpairs_n \arrow[r,hookrightarrow, "\tau\times\rho"]
				& \pairs_n \arrow[r,hookrightarrow]
				&\fpairs_{2n} \arrow[r,hookrightarrow, "\tau\times\rho"]
				&\pairs_{2n}
			\end{tikzcd}
		\end{center}
		where the map $\pairs_n\to\fpairs_{2n}$ is $\fRSK\circ F\circ\RSK^{-1}$. 
		
		Let $G:\pairs_n\to\pairs_{2n}$ be the map that adds $n$ to each entry of the recording tableau. From definitions we have $\RSK\circ F= G\circ \RSK$ so $G$ is contextualized by the commutative diagram. Also let $(\tau\times\rho)^\dagger$ be the left inverse for $\tau\times\rho$ from Lemma \ref{lem:subset}. 
	
		For $(S,T)\in\fpairs_n$, say we increment the entries of $T$ by $n$ and shift each entry in both fillings up by $n$ rows. Call the result $(S^{\uparrow},T^\uparrow)$. Each entry in the first column of $T^\uparrow$ is still in the row matching its entry, so $T^\uparrow\in\rSSAF_{2n}$. The only non-trivial property to check to ensure $S^\uparrow \in \SSKT_{2n}$ is that there are no Type II co-inversion triples involving the basement column, but this follows from Lemma \ref{lem:ledge}. Notice that $(S^{\uparrow},T^\uparrow)$ have the same column sets as 
		\[
		(\tau\times\rho)^\dagger\circ G\circ (\tau\times\rho)(S,T)
		\]
		and so must be the object we get by following the diagram from $\fpairs_n$ to $\fpairs_{2n}$.
		
		Now instead take a biword $W=\begin{pmatrix} i_1&\cdots&i_{\ell}\\ j_1&\cdots&j_{\ell} \end{pmatrix}$ not associated with a lower triangular matrix, so $i_t<j_t$ for some minimal $t$. When we insert $j_t$ into the SSKT during the operation $\fRSK(F(W))$ it will be the largest entry thus far and therefore be placed in the first column of row $i_t+n$. Letting $(S',T')=\fRSK(F(W))$ this implies that the $S'(1,i_t+n)\ge j_t$. Then there is no $S\in\SSKT_n$ for which $S'=S^\uparrow$, as that would require $\hat S(1,i_t)\ge j_t>i_t=\hat S(0,i_t)$ hence $\maj(\hat S)>0$.
		
		The image of $\tau\times\rho:\fpairs_n\to \pairs_n$ is therefore contained in, and equal to, the image of $\RSK|_{\mathcal L_n}:\mathcal L_n\to \pairs_n$. We conclude $\fRSK$ is a bijection.
	\end{proof}
	\begin{remark}
		While $\fRSK$ is perhaps best thought of as a restriction of $\RSK$ as suggested by Theorem \ref{thm:rsk}, the proof of Theorem \ref{thm:bijection} is interesting in that it implies $\fRSK$ is actually equivalent to $\RSK$. Indeed, we have \[\RSK=G^{-1}\circ(\tau\times\rho)\circ \fRSK\circ F\]
		(abusing notation so that $G^{-1}$ is really just the obvious left inverse for $G$).
	\end{remark}
	
	Finally, our motivating application for this bijection was a new proof of Theorem \ref{thm:cauchy}.
	\begin{proof}[Proof of Theorem \ref{thm:cauchy}]
		We can see
		\[
		\prod_{1\le j\le i\le n}(1-x_iy_j)^{-1}=\sum_{L\in\mathcal L_n}x^{\mathrm{row}(L)}y^{\mathrm{col}(L)}.
		\]
		By Theorem \ref{thm:mac} we also have
		\[
		\sum_{\comp{a}\in \mathbb N^n}E_{\comp a}(X;\infty,\infty)E_{\comp a}(Y;0,0)=\sum_{(S,T)\in \fpairs_n}x^{\wt(T)}y^{\wt(S)}.
		\]
		The right hand sides are equal by Theorem \ref{thm:bijection}.
	\end{proof}
	
	%
	\section{Schubert Expansions}\label{sec:schubert}
	%
	
	Another generalization of the Schur polynomials are the Schubert polynomials $\schubert_w$ \cite{LS82}. They are indexed by permutations $w\in S_{\infty}$ that fix all but finitely many positive integers and represent Schubert classes in the cohomology of the complete flag variety. A permutation $v$ is a \newword{$k$-grassmanian} if $v(i)<v(i+1)$ whenever $i\ne k$. The $k$-grassmannians correspond to partitions $\lambda$ of at most $k$ parts by taking $\lambda_{k+1-i}=v(i)-i$, and we write $v=v(\lambda,k)$. We recover the Schur polynomials as
	\[
	s_{\lambda}(x_1,\ldots,x_k)=\schubert_{v(\lambda,k)}.
	\]
	One can find a discussion of Schubert polynomials in \cite{Ful97}. Using a Pieri rule we will see that the complete flagged homogeneous polynomials are a positive sum of Schubert polynomials.
	
	The \newword{$k$-Bruhat order} defined by Bergeron and Sottile \cite{BS98} is the transitive closure of the cover relation $w\le_k wt_{i,j}$ whenever $\ell(w t_{i,j}) = \ell(w) + 1$ and $i\le k<j$. The following Pieri rule was conjectured by Bergeron and Billey \cite{BB93}, before being proved geometrically by Sottile \cite{Sot96} and combinatorially by Kogan and Kumar \cite{KK02}. 
	
	\begin{theorem}[\cite{Sot96}]
		For $u$ a permutation and $m,k$ positive integers, we have
		\[
		\schubert_u \cdot \schubert_{v((m),k)} = \sum_{\substack{ u \le_k w \\ w = u t_{i_1,j_1} \cdots t_{i_m,j_m} \\ j_1,\ldots, j_m \ \text{distinct} }} \schubert_{w} .
		\]
		\label{thm:pieri}
	\end{theorem}

	We saw the $\h$-basis does not contain the complete homogeneous symmetric polynomials and so is not truly a multiplicative basis.	Its structure constants can even be negative. For instance
	\[ \h_{(0,1)}^2 = \h_{(0,2)} + \h_{(1,1)} - \h_{(2,0)}. \]
	
	Interestingly, it follows from Theorem~\ref{thm:pieri} that the product of $\h$'s remains Schubert positive.
	
	\begin{corollary}
		Given compositions $\comp a$ and $\comp b$, we have
		\[ \h_{\comp a} \h_{\comp b} = \sum_{w} c_{{\comp a},{\comp b}}^{w} \schubert_w \]
		where $c_{{\comp a},{\comp b}}^{w}$ are nonnegative integers.
	\end{corollary}
	\begin{proof}
		For a weak composition $0^k\times(m)$ with at most one nonzero part, we have
		\[
		\h_{0^k\times(m)}=s_{(m)}(x_1,\ldots,x_k)=\schubert_{v((m),k)}.
		\]
		Any product of complete flagged homogeneous polynomials is therefore a product of such Schubert polynomials which has a non-negative Schubert expansion by Theorem \ref{thm:pieri}.
	\end{proof}

	We will explicitly describe the Schubert expansion of a single $\h_{\comp a}$ polynomial. To reduce notation we say $w/u$ is a \newword{horizontal $m$-strip in $k$-Bruhat order} if $u\le_k w$ and there exists a saturated chain in $k$-Bruhat order of length $m$, say exemplified by $w = u t_{i_1,j_1}\cdots t_{i_m,j_m}$, where $j_1,\ldots,j_m$ are all distinct.

	\begin{theorem}\label{thm:h-schubert}
		For a weak composition $\comp b = (\comp b_1,\ldots,\comp b_m)$, we have
		\[ \h_{\comp b} = \sum_{w} C_{w,{\comp b}} \schubert_w, \]
		where $C_{w,{\comp b}}$ is the number of sequences $\mathrm{id} = w^{(0)},\ldots,w^{(m)} = w$ such that each $w^{(k)}/w^{(k-1)}$ is a horizontal $(\comp b_k)$-strip in $k$-Bruhat order.
	\end{theorem}
	
	\begin{proof}
		We proceed by induction on $m$. For $m=0$ (i.e. $\comp b$ has no nonzero parts) we have
		\[
		\h_{\comp b}=1=s_\emptyset=\schubert_{\mathrm{id}}.
		\]
		The only sequence enumerated by a $C_{w,\comp b}$ is $\mathrm{id}=w^{(0)}$ so $C_{\mathrm{id},\comp b}=1$ is the only nonzero value. This proves the base case.
		
		Now take $m>0$ and assume the result holds for weak compositions of length $m-1$. We have
		\[
		\h_{\comp b}=\h_{(\comp b_1,\ldots,\comp b_{m-1})}s_{(\comp b_m)}(x_1,\ldots,x_m)\\
		=\sum_{w} C_{w,{(\comp b_1,\ldots,\comp b_{m-1})}} \schubert_w \schubert_{v((\comp b_m),m)}
		\]
		and applying Theorem \ref{thm:pieri} yields the result by induction.  
	\end{proof}

	%
	\section{The Inverse Kostka Analogue}\label{sec:inverse}
	%
	
	Eğecioğlu and Remmel found a signed combinatorial formula for the entries of the classical inverse Kostka matrix \cite{ER90}. We generalize their notion of a rim hook tabloid from the context of Ferrers diagrams to key diagrams in order to analogously expand the key polynomials into the basis of complete flagged homogeneous polynomials.
	
	First let us reconstruct the definition of a rim hook. \textbf{Young's lattice} is the set of partitions with length at most $n$ partially ordered by the relation $\lambda\preceq_Y \mu$ whenever $\lambda_i\le\mu_i$ for all $1\le i\le n$. We will say two cells are \newword{connected} if they are either vertically or horizontally adjacent. Taking the transitive closure of the ``connected'' relation decomposes a diagram into equivalence classes, which we call \newword{connected components}. We say the diagram is connected if there is a unique connected component.
	\begin{definition}\label{def:rim-hook}
		A \newword{rim hook} $R$ of a partition $\mu$ is a subset of the English Ferrers diagram $\D(\rev(\mu))$ such that
		\begin{enumerate}
			\item $R$ is connected,
			\item $\D(\rev(\mu) )\setminus R=\D(\rev(\lambda))$ for some $\lambda\preceq_Y \mu$, and
			\item $R$ does not contain a quadruple of cells $(c,r),(c+1,r),(c,r+1),(c+1,r+1)$.
		\end{enumerate}
	\end{definition}

	\begin{figure}[ht]
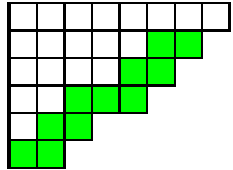

		\vline
		\begin{ytableau}
			~&~&~&~&~&~&~&~\\
			~&~&~&~&~&*(green)~&*(green)~\\
			~&~&~&~&*(green)~&*(green)~\\
			~&~&*(green)~&*(green)~&*(green)~\\
			~&*(green)~&*(green)~\\
			*(green)~&*(green)~
		\end{ytableau}
		\caption{\label{fig:rimhook}A rim hook of $\lambda=(8,7,6,5,3,2)$.}
	\end{figure}

	Now the replacement for Young's lattice on weak compositions is the following partial order.
	\begin{definition}[\cite{AW22}]\label{def:key-poset}
		The \textbf{key poset} on weak compositions of length at most $n$ is the partial order defined by $\comp a\preceq\comp b$ whenever
		\begin{enumerate}[(i)]
			\item $\comp a_i\le\comp b_i$ for all $1\le i\le n$, and
			\item if $\comp a_i>\comp a_j$ for some $1\le i<j\le n$ then $\comp b_i>\comp b_j$.
		\end{enumerate}
	\end{definition}
	As an example
	\[
	(0,6,0,1,2,8,4)\preceq(3,7,0,2,5,8,6)\quad\text{but}\quad (0,6,0,1,5,8,2)\not\preceq(3,7,0,2,5,8,6).
	\]

	We say that two cells are \newword{weakly connected} if they share a column or lie in adjacent columns with the cell on the left in a weakly higher row (cf. attacking cells). As before, this generates an equivalence relation on the cells in a diagram, whose classes we call \newword{weakly connected components}, and we say the diagram is weakly connected if there is a unique weakly connected component. An example of the concept is illustrated in Figure \ref{fig:connected}.
	
	\begin{figure}[ht]
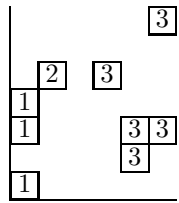

		\vline
		\begin{ytableau}
			\none&\none&\none&\none&\none&3\\
			\none\\
			\none&2&\none&3\\
			1&\none&\none&\none\\
			1&\none&\none&\none&3&3\\
			\none&\none&\none&\none&3 \\
			1 \\
			\hline
		\end{ytableau}
		\caption{\label{fig:connected}A diagram with enumerated weakly connected components.}
	\end{figure}
	
	Rim hooks are now generalized by the following definition.
	\begin{definition}\label{def:snake}
		A \textbf{snake} $S$ of a weak composition $\comp b$ is a subset of the key diagram $\D(\comp b)$ such that
		\begin{enumerate}
			\item $S$ is weakly connected,
			\item $\D(\comp b)\setminus S=\D(\comp a)$ for some weak composition $\comp a\preceq\comp b$, and
			\item $S$ does not contain a triple of cells $(c,s),(c+1,s),(c+1,r)$ with $r<s$.
		\end{enumerate}
	\end{definition}

	\begin{figure}[ht]
		\vline
		\begin{ytableau}
			~&~&~&~&*(green)~&*(green)~\\
			~&~&~&~&~&~&~&~\\
			~&~&*(green)~&*(green)~&*(green)~\\
			~&*(green)~\\
			\none\\
			~&~&~&~&~&~&*(green)~\\
			*(green)~&*(green)~&*(green)~\\
			\hline
		\end{ytableau}
		\caption{\label{fig:snake}A snake of $\comp b=(3,7,0,2,5,8,6)$.}
	\end{figure}
	
	The next result shows that the generalization is a natural one.
	\begin{proposition}\label{prop:hook}
		Let $\mu$ be a partition. The snakes of $\rev(\mu)$ are exactly the rim hooks of $\mu$.
	\end{proposition}
	\begin{proof}
		Suppose $\lambda\preceq_Y\mu$ for a partition $\lambda$. Then $\rev(\lambda)\preceq\rev(\mu)$ as (ii) in Definition \ref{def:key-poset} is vacuously satisfied. Conversely, if $\comp b\preceq \rev(\mu)$ then the fact that the parts of $\rev(\mu)$ are weakly increasing implies the same is true for $\comp b$. Thus $\comp b=\rev(\lambda)$ for a partition $\lambda$. That is, $\lambda\mapsto \rev(\lambda)$ is an injective morphism of posets and its image is closed under ``going down'' in the poset.
		
		Now say that $R=\D(\rev(\mu))\setminus\D(\rev(\lambda))$ is a rim hook of $\mu$ with $\lambda\preceq_Y\mu$. A connected pair of cells is also weakly connected, so $R$ is weakly connected. We have also seen that $\rev(\lambda)\preceq\rev(\mu)$. 
		
		Suppose $R$ contains a triple $(c,s),(c+1,s),(c+1,r)$ of cells with $r<s$. Then $(c,s),(c+1,s)\notin \D(\rev(\lambda))$ which implies $(c,s-1),(c+1,s-1)\notin \D(\rev(\lambda))$ since the parts of $\rev(\lambda)$ weakly increase. However, $(c,s-1),(c+1,s-1)\in \D(\rev(\mu))$ since $\rev(\mu)_s\ge\rev(\mu)_r\ge c+1$. In particular $(c,s),(c+1,s),(c,s-1),(c+1,s-1)$ all belong to $R$ which contradicts that $R$ is a rim hook. We must conclude that $R$ is a snake of $\rev(\mu)$.
		
		On the other hand, suppose $S$ is a snake of $\rev(\mu)$. We have seen that (2) in Definition \ref{def:snake} implies $S=\D(\rev(\mu))\setminus \D(\rev(\lambda))$ for a partition $\lambda\preceq_Y\mu$. Condition (3) in Definition \ref{def:snake} is actually stronger than (3) in Definition \ref{def:rim-hook} so the latter is satisfied. It remains to argue $S$ is connected.
		
		Say $u,v\in S$ form a weakly connected pair. Suppose first that they lie in the same column with $u$ below $v$. Since the parts of $\rev(\mu)$ weakly increase and $u\in\D(\rev(\mu))$ we have that every cell above $u$ in the same column is in $\D(\rev(\mu))$. Similarly, since $v\notin\D(\rev(\lambda))$ every cell below $v$ in the same column is not in $\D(\rev(\lambda))$. Then there is a connected vertical strip of cells in $S$ between $u$ and $v$ putting them in the same connected component.
		
		Now suppose $u,v$ lie in adjacent columns, say with $v$ left of, and weakly above $u$. Take the cell $w$ immediately left of $u$. As before, $v\notin\D(\rev(\lambda))$ implies $w\notin\D(\rev(\lambda))$. So $w\in S$, $(w,u)$ is connected, and $w$ and $v$ share a connected component by the above argument. It follows that $u$ and $v$ lie in the same component. Therefore the unique weakly connected component of $S$ is also the unique connected component of $S$, so $S$ is a rim hook.
	\end{proof}
	
	A snake is \textbf{special} if it is empty or contains the lowest cell in column 1 of the key diagram. The objects that index our expansion into the $\h$ basis are the special snake tabloids.
	
	\begin{definition}\label{def:snake-tabloid}
		A \textbf{special snake tabloid} $(S_1,\ldots,S_n)$ of shape $\comp b$ with at most $n$ parts is a decomposition $\D(\comp b)=S_1\sqcup\cdots\sqcup S_n$ such that each $S_i$ is a special snake of $\D(\comp b)\setminus(S_1\sqcup\cdots\sqcup S_{i-1})$, with $S_i=\emptyset$ if and only if row $i$ of this subdiagram is empty.
	\end{definition}
	
	The \textbf{weight} of a special snake tabloid $U=(S_1,\ldots,S_n)$ is the weak composition
	\begin{equation}
		\wt(U)=(|S_1|,\ldots, |S_n|).
	\end{equation}
	
	The \textbf{height} of a snake $S$, denoted $\height(S)$, is the number of rows in which $S$ contains a cell, except where $S=\emptyset$ in which case the height is 1. The \textbf{sign} of a snake is 
	\begin{equation}
		\sgn(S)=(-1)^{\height(S)-1},
	\end{equation}
	and the sign of a special snake tabloid is the product of signs of its snakes.
	
	\begin{figure}[ht]
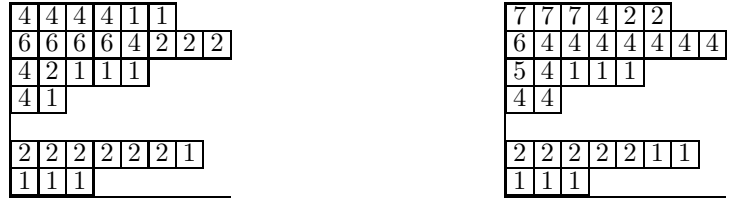

		\vline
		\begin{ytableau}
			4&4&4&4&1&1\\
			6&6&6&6&4&2&2&2\\
			4&2&1&1&1\\
			4&1\\
			\none\\
			2&2&2&2&2&2&1\\
			1&1&1\\
			\hline
		\end{ytableau}\hspace{100pt}
		\vline
		\begin{ytableau}
			7&7&7&4&2&2\\
			6&4&4&4&4&4&4&4\\
			5&4&1&1&1\\
			4&4\\
			\none\\
			2&2&2&2&2&1&1\\
			1&1&1\\
			\hline
		\end{ytableau}
		\caption{\label{fig:tabloids}Two special snake tabloids of shape $(3,7,0,2,5,8,6)$, with respective weights $(10,10,0,7,0,4,0)$ and $(8,7,0,11,1,1,3)$ and signs $-1$ and $+1$.}
	\end{figure}
	
	We can now state our theorem.
	\begin{theorem}\label{thm:snakes}
		For a weak composition $\comp b$ we have
		\begin{equation}\label{eq:snake}
			\key_{\comp b}=\sum_{\comp a\in\mathbb N^n}\tilde K^{-1}_{\comp a\comp b}\h_{\comp a}
		\end{equation}
		where
		\[
		\tilde K^{-1}_{\comp a\comp b}=\sum_U \sgn(U)
		\]
		summed over special snake tabloids $U$ of weight $\comp a$ and shape $\comp b$.
	\end{theorem}
	The notation $\tilde K^{-1}_{\comp a\comp b}$ reflects that these coefficients, and the coefficients $\tilde K_{\comp a\comp b}$ define inverse transition matrices.
	
	It is interesting to compare Theorem \ref{thm:snakes} with Eğecioğlu and Remmel's result. In light of Proposition \ref{prop:hook} it can be stated as follows.
	\begin{theorem}[\cite{ER90}]\label{thm:ER}
		For a partition $\mu$ we have
		\begin{equation}\label{eq:ER}
			s_\mu(x_1,\ldots,x_n)=\sum_\lambda K^{-1}_{\lambda\mu} h_\lambda(x_1,\ldots,x_n)
		\end{equation}
		summed over partitions, and with
		\[
		K^{-1}_{\lambda\mu}=\sum_U \sgn(U)
		\]
		summed over special snake tabloids $U$ of weight $\lambda$ and shape $\mu$.
	\end{theorem}
	Applying Theorem \ref{thm:snakes} to $\key_{\rev(\mu)}=s_\mu(x_1,\ldots,x_n)$ we see that the same objects index the expansion of $s_\mu$ into both bases $\{ h_\lambda \}$ and $\{ \h_\comp a \}$. This is in spite of the fact that the former basis is not contained in the latter, but consistent with Proposition \ref{prop:stable_h}. 
	
	There is no known cancellation-free formula for the inverse Kostka coefficients $K^{-1}_{\lambda\mu}$. Theorem \ref{thm:snakes} is not cancellation-free either as seen by Figure \ref{fig:cancel}, however it \textit{is} cancellation free in the special case of Schur polynomials.
	\begin{figure}[ht]
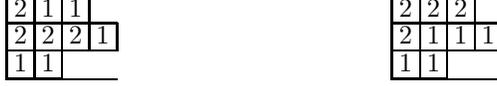

		\vline
		\begin{ytableau}
			2&1&1\\
			2&2&2&1\\
			1&1\\
			\hline
		\end{ytableau}\hspace{100pt}
		\vline
		\begin{ytableau}
			2&2&2\\
			2&1&1&1\\
			1&1\\
			\hline
		\end{ytableau}
		\caption{\label{fig:cancel}Special snake tabloids with the same weight and shape but opposite signs.}
	\end{figure}
	\begin{theorem}\label{thm:cancelfree}
		For a partition $\mu$, the formula
		\begin{equation}\label{eq:nocancel}
		\tilde K^{-1}_{\comp a\rev(\mu)}=\sum_U \sgn(U)
		\end{equation}
		summed over special snake tabloids $U$ of weight $\comp a$ and shape $\rev(\mu)$ is cancellation-free, i.e. does not contain both positive and negative terms.
	\end{theorem}
	\begin{proof}
		For a partition $\lambda=(\lambda_1,\ldots,\lambda_\ell,0,0,\ldots)$ we claim that a special snake $S$ of $\rev(\lambda)$ is completely determined by $\lambda$ and the size of $S$. Since $S$ is a rim hook by Proposition \ref{prop:hook} we work with Definition \ref{def:rim-hook}. If $|S|=0$ the claim is obvious so we may assume $(1,n+1-k)\in S$.
		
		Say $\D(\rev(\mu))=\D(\rev(\lambda))\setminus S$ with $\mu\preceq_Y \lambda$. Suppose $(r,c)\in S$ with $c<\rev(\lambda)_{r-1}$. We have $\rev(\mu)_r< c$ which implies $\rev(\mu)_{r-1}<c$. Then $(r,c),(r,c+1),(r-1,c),(r-1,c+1)\in S$ which contradicts Definition \ref{def:rim-hook}. So we must have $c\ge \rev(\lambda)_{r-1}$ for any $(r,c)\in S$ with $r>1$.
		
		Let \[V=\{(r,c)\in\D(\rev(\lambda))\mid r=1 \text{ or } c\ge \rev(\lambda)_{r-1}  \}\] and take a cell $(r,c)\in V$. If $(r,c+1)\in V$ then $c< \rev(\lambda)_{r}$ which implies $(r+1,c)\notin V$. If $(r-1,c)\in V$ then $c\le \rev(\lambda)_{r-1}$ implies $(r,c-1)\notin V$. Then taking $U$ to be the connected component of $V$ containing $(1,n+1-\ell)$ we can  write $U=\{u_1,u_2,\ldots,u_m \}$ where each pair $(u_i,u_{i+1})$ is connected with $u_{i+1}$ to the right or above $u_i$. Note that if $m>2$ then the row and column differences of $u_1$ and $u_m$ sum to $m$ and therefore $(u_1,u_m)$ is not connected.
		
		It must be the case that $u_1=(1,n+1-\ell)$ since this is the leftmost cell in the lowest row of $\D(\rev(\lambda))$. Now $S$ is a connected subset of $U$ containing $u_1$ which can be nothing but $S=\{u_1,u_2,\ldots,u_k \}$ for $k$ the cardinality of $S$. This proves the claim that $S$ is determined by its size and $\lambda$.
		
		Given a special snake tabloid $U=(S_1,\ldots, S_n)$ of shape $\rev(\mu)$ we know by induction that each $S_i$ is a rim hook of a Ferrers diagram $\D(\rev(\lambda))\setminus (S_1\sqcup\cdots\sqcup S_{i-1})$. It therefore follows from the claim that $U$ is completely determined by $\mu$ and $\wt(U)$. Not only is the sum \eqref{eq:nocancel} cancellation-free: it has at most one term.
	\end{proof}
	
	From Theorem \ref{thm:cancelfree} we then see that the entirety of cancellation that occurs in Theorem \ref{thm:ER} is explained as a result of distinct complete flagged homogeneous polynomials stabilizing to the same complete homogeneous symmetric function. For instance, Eğecioğlu and Remmel give the tabloids in Figure \ref{fig:non-cancel} as an example of objects whose terms cancel in equation \eqref{eq:ER} \cite{ER90}. This occurs because their sorted weights yield the same partition. However the weights are necessarily distinct as weak compositions so the terms do not cancel in equation \eqref{eq:snake}; just their stable limits do.

	\begin{figure}[ht]
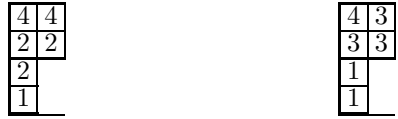

		\vline
		\begin{ytableau}
			4&4\\
			2&2\\
			2\\
			1\\
			\hline
		\end{ytableau}\hspace{100pt}
		\vline
		\begin{ytableau}
			4&3\\
			3&3\\
			1\\
			1\\
			\hline
		\end{ytableau}
		\caption{\label{fig:non-cancel}Special snake tabloids whose terms cancel in equation \eqref{eq:ER} but not in equation \eqref{eq:snake}. }
	\end{figure}
	
	Further applications for snakes are outside the scope of this paper, however it is worth considering that these objects could extend definitions and results that traditionally rely on rim hooks. For instance, the Murnaghan-Nakayama rule can be used to write the power sum symmetric polynomials in terms of rim hooks and the Schur basis \cite{Mur37,Nak40}. Replacing rim hooks with snakes and Schur polynomials with key polynomials defines what may be an interesting power sum analogue. Additionally, the polynomials of Lascoux, Leclerc, and Thibon are defined in terms of rim hooks and may also be candidates for generalization \cite{LLT97}.
	
	The remainder of the section builds up to the proof of Theorem \ref{thm:snakes}. Given a snake $S$ of $\comp b$, our next lemmas require the set
	\begin{equation}
		\mathcal G(S)=\{ T:\D(\comp b)\to[n] \mid T|_{\D(\comp b)\setminus S}\in\SSKT \text{ and } T|_S=1  \}.
	\end{equation}
	Notice that the generating function of this set can be written
	\[
	\sum_{T\in\mathcal G(S)}x^{\wt(T)}=x_1^{|S|}\key_{\sh(\D(\comp b)\setminus S)}
	\]
	where $\sh(\D(\comp b)\setminus S)$ is the weak composition indexing the key diagram $\D(\comp b)\setminus S$
	
	\begin{figure}[ht]
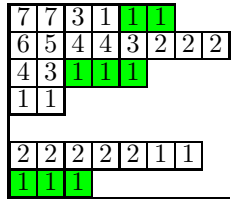

		\vline
		\begin{ytableau}
			7&7&3&1&*(green)1&*(green)1\\
			6&5&4&4&3&2&2&2\\
			4&3&*(green)1&*(green)1&*(green)1\\
			1&1\\
			\none\\
			2&2&2&2&2&1&1\\
			*(green)1&*(green)1&*(green)1\\
			\hline
		\end{ytableau}
		\caption{\label{fig:almost-sskt}A filling $T\in\mathcal G(S)$ with $S$ highlighted.}
	\end{figure}
	
	\begin{lemma}\label{lem:1-snake}
		Let $S$ be a snake of $\D(\comp b)$. If $T\in \mathcal G(S)$ then $\maj(\hat T)=\coinv(\hat T)=0$.
	\end{lemma}
	\begin{proof}
		Since $T|_S$ is identically 1 and $T|_{\D(\comp b)\setminus S}$ is an SSKT we must have $\maj(\hat T)=0$.
		
		Say $\D(\comp a)=S\setminus \D(\comp b)$ so that $\comp a\preceq \comp b$. Take rows $r<s$ such that $\comp b_r\le \comp b_s$. It follows that $\comp a_r\le \comp a_s$ as well. Then any co-inversion triple of $\hat T$ within these rows must have an entry in $S$, else it would also be a co-inversion triple in $\hat T|_{\D(\comp a)}$. However, $\comp a_r\le \comp a_s$ implies that any cell in row $s$ of $S$ does not share a column with a cell in row $r$ of $\D(\comp a)$. So any (necessarily Type II) triple of $T$ whose top right cell is in $S$ must also have its bottom cell in $S$. The triple then contains two entries 1 and is not a co-inversion triple.
		
		Now take rows $r<s$ with $\comp b_r>\comp b_s$. If $\comp a_r\le \comp a_s$ then by Lemma \ref{lem:ledge} we have $T(c,r)<T(c,s)$ whenever $c\le \comp a_r$. Since $T(c,r)=1$ for all $\comp a_r<c\le \comp b_r$, it follows that $T(c,r)\le T(c,s)$ for all $c\le \comp b_s$. In this case there are no co-inversion triples of $T$ within these rows.
		
		If instead $\comp a_r>\comp a_s$, then any (Type I) co-inversion triple of $T$ within these rows must have entries $1<i<j$ with the entry 1 in $S$. The premise $\comp a_r>\comp a_s$ then further implies that the cell of $i$ is in $S$, so $i=1$ which is a contradiction.
	\end{proof}
	
	\begin{definition}\label{def:snake-attack}
		Given a snake $S$ of $\D(\comp b)$ and $T\in \mathcal G(S)$, we say an ordered pair of cells $x,y\in\D(\comp b)$ is an \newword{$S$-attack} if
		\begin{enumerate}[(a)]
			\item $x$ and $y$ are attacking cells with $y$ either above $x$ or in the column to its right,
			\item $x\in S$,
			\item both cells have entry 1, and
			\item if $x$ and $y$ lie in distinct columns then $S$ does not contain the cell immediately left of $y$.
		\end{enumerate}
	\end{definition}
	Condition (d) is a bit of a technical point, but an $S$-attack is approximately just a violation of the non-attacking condition for $T$ with some symmetry-breaking stipulations.
	
	\begin{lemma}\label{lem:snake-attack}
		Let $S$ be a special snake of $\D(\comp b)$. If $T\in \mathcal G(S)$ is not an SSKT then $T$ contains a $S$-attack.
	\end{lemma}
	\begin{proof}
		By Lemma \ref{lem:1-snake}, $\maj(\hat T)=\coinv(\hat T)=0$ but $T$ is not an SSKT so it must contain a pair of attacking cells with entry 1 and at least one of those cells in $S$.
		
		If $S$ is composed only of the first row of $\D(\comp b)$. Any cell $y$ with entry 1 attacking a cell in $S$ must in fact lie in the same column as some cell $x$ in $S$ so this pair is an $S$-attack.
		
		If $S$ contains cells in multiple rows, let $y$ be the leftmost, then lowest cell in $S$ that is not in row 1. In order for $S$ to be weakly connected, $y$ must share a column with a cell $x$ in row 1 and again this pair is an $S$-attack.
	\end{proof}
	
	The central step in our proof of the $\key$ to $\h$ expansion is showing the existence of the following involution on the set
	\begin{equation}
	\mathcal F=\{ (S,T) \mid S \text{ a special snake of } \D(\comp b),\  T\in \mathcal G(S) \text{ and } T\notin\SSKT \}.
	\end{equation}
	when $\comp b_1>0$. 
	
	\begin{definition}\label{def:involution}
		With $\comp b_1>0$ and $(S,T)\in\mathcal F$ define $S'$ as follows. Take $x$ to be the rightmost, then topmost cell that occurs as the first cell in an $S$-attack. Let $y$ be the rightmost, then lowest cell such that $x,y$ is an $S$-attack. Let $B(y)$ be the set containing $y$ and all cells in the same row to its right in $\D(\comp b)$. Then $S'$ is defined to be the symmetric difference of $S$ and $B(y)$. We also write $\iota(S,T)=(S',T)$.
	\end{definition}
	\begin{remark}
		A choice of $x$ as in Definition \ref{def:involution} exists by Lemma \ref{lem:snake-attack}, so the procedure is well-defined.
	\end{remark}

	We claim $\iota$ is a sign-reversing involution on $\mathcal F$. The content of the proof is in several intermediate results. The next lemma should help clarify the nature of the procedure.
	
	\begin{figure}[ht]
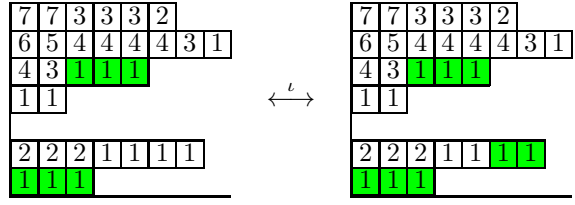

		\vline
		\begin{ytableau}
			7&7&3&3&3&2\\
			6&5&4&4&4&4&3&1\\
			4&3&*(green)1&*(green)1&*(green)1\\
			1&1\\
			\none\\
			2&2&2&1&1&1&1\\
			*(green)1&*(green)1&*(green)1\\
			\hline
		\end{ytableau} $\quad\overset{\iota}{\longleftrightarrow}\quad$ \vline
		\begin{ytableau}
			7&7&3&3&3&2\\
			6&5&4&4&4&4&3&1\\
			4&3&*(green)1&*(green)1&*(green)1\\
			1&1\\
			\none\\
			2&2&2&1&1&*(green)1&*(green)1\\
			*(green)1&*(green)1&*(green)1\\
			\hline
		\end{ytableau}
		\caption{\label{fig:involution} The map $\iota$.}
	\end{figure}
	
	\begin{lemma}\label{lem:sym-diff}
		In the setting of Definition \ref{def:involution}, either $B(y)\subset S$ or $B(y)\cap S=\emptyset$. Equivalently, either $S'=S\setminus B(y)$ or $S'=S\cup B(y)$.
	\end{lemma}
	\begin{proof}
		Suppose $z$ is the leftmost cell in $S\cap B(y)$. If $z$ is not in column 1 then in order for $S$ to be weakly connected its highest cell $w$ in the column left of $z$ must be weakly above some cell $v\in S$ in the same column as $z$. If $w$ does not attack $z$, then $v$ is distinct from $z$ and attacks $z$. Either way, we see that $z$ must be $y$ so $B(y)\subset S$.
	\end{proof}

	\begin{lemma}\label{lem:inv-snake}
		In the setting of Definition \ref{def:involution}, $S'$ is a special snake of $\comp b$.
	\end{lemma}
	\begin{proof}
		By Lemma \ref{lem:sym-diff}, either $S'=S\cup B(y)$ or $S'=S\setminus B(y)$. We first show that either set is weakly connected.
		
		Both $S$ and $B(y)$ are individually weakly connected. Moreover $y$ is weakly connected to $x$ in $S$, so $S\cup B(y)$ is weakly connected. In the case $B(y)\subset S$, suppose $S\setminus B(y)$ fails to be weakly connected. There must be a cell $z\in B(y)$ attacking a cell $w\in S\setminus B(y)$ such that $x$ and $w$ do not share a weakly connected component of $S\setminus B(y)$. If the cell $w'$ immediately left of $w$ is in $S$, then $w'$ must lie in a column strictly left of $x$. The pair of $w'$ and $z$ then constitute an $S$-attack contradicting the choice of $x$. If there is no such $w'$ in $S$ then $w$ and $z$ in some order constitute an $S$-attack which still contradicts the choice of $x$.
		
		Next we show $\D(\comp b)\setminus S'=\D(\comp a')$ for some $\comp a'\preceq\comp b$. Let $\D(\comp a)=\D(\comp b)\setminus S$ so that $\comp a\preceq \comp b$. When $S\cap B(y)=\emptyset$ it is clear that $\D(\comp b)\setminus (S\cup B(y))=\D(\comp a)\setminus B(y)$ is indeed a key diagram $\D(\comp a')$. Say we have $r<s$ with $\comp a_r\ge \comp a_s$. Since $T|_{\D(\comp a)}$ is a SSKT we can apply Lemma \ref{lem:ledge} to see that no entry 1 in row $s$ of $T|_{\D(\comp a)}$ can share a column with a cell in row $r$ of $\D(\comp a)$. Every entry of $B(y)$ must be 1, since $y$ has entry 1 $\maj(\hat T)=0$. Thus, $B(y)$ does not lie in row $s$. It follows that $\comp a'_r\ge \comp a'_s$, so $\comp a'\preceq\comp a\preceq \comp b$.
		
		If $B(y)\subset S$ then the cell immediately left of $y$ cannot be contained in $S$ without violating either (3) in Definition \ref{def:snake} or (d) in Definition \ref{def:snake-attack}. Thus $\D(\comp b)\setminus(S\setminus B(y))$ is also a key diagram $\D(\comp a')$. Take row indices $r<s$ with $\comp b_r\ge \comp b_s$. Say there is a cell $(c,s)\in S\setminus B(y)$. If $(r,c)\in \D(a)$ then we must have $(r,c)\in S$ since $\comp a\preceq \comp b$. Thus, we will have $\comp a'_r\ge \comp a'_s$ unless $(r,c)\in B(y)$. In the latter case, $(r,c)$ and $(s,c)$ constitute an $S$-attack with $(r,c)$ contradicting the choice of $x$. Therefore $\comp a'\preceq \comp b$.
		
		We finally check (3) in Definition \ref{def:snake}. Since $S$ is a snake, $S\setminus B(y)$ maintains compliance. Suppose $S\cap B(y)=\emptyset$ and we have $z\in B(y)$ with some $w\in S$ sharing a column. If $w$ is below $z$ then $w,z$ is an $S$-attack forcing $z=y$. A cell immediately left of $z$ can then not be in $B(y)$ or $S$ which means it will not violate (3) together with $w$ and $z$. If $w$ is above $z$ and there is a cell $z'$ immediately right of $z$ then $w,z'$ would be an $S$-attack in contradiction of our choice of $x$ and $y$. If such $z'$ does not exists, then the fact that $w\in S$ and $z\notin S$ contradicts $\comp a\preceq \comp b$. Therefore $S\cup B(y)$ must satisfy (3).
		
		We have shown that $S'$ is a snake. It is in fact a special snake since $y$ being the second cell in an $S$-attack precludes it from being $(1,1)$. 
	\end{proof}

	\begin{lemma}\label{lem:inv-image}
		In the setting of Definition \ref{def:involution}, we have $(S',T)\in\mathcal F$.
	\end{lemma}
	\begin{proof}
		Since $S'$ is a special snake by Lemma \ref{lem:inv-snake}, we just need to show $T\in\mathcal G(S')$. We have $T|_{S'}=1$ by construction so we are further reduced to the claim that $T|_{\D(\comp b)\setminus S'}$ is an SSKT. 
		
		Say $\D(\comp a)=\D(\comp b)\setminus S$ and $\D(\comp a')=\D(\comp b)\setminus S'$. If $\D(\comp a')= \D(\comp a)\setminus B(y)$ then $\hat T|_{\D(\comp a')}$ will remain non-attacking with weakly decreasing rows left to right. Suppose $T|_{\D(\comp a')}$ contains a Type II (resp. Type I) co-inversion triple $i<j<k$ as in Fig. \ref{fig:inv}. This is only possible if the row $s$ containing $i,k$ is strictly (resp. weakly) shorter than the row $r$ containing $j$ in $T|_{\D(\comp a)}$. Since $T|_{\D(\comp a)}$ is an SSKT, the Type I case is impossible by Lemma \ref{lem:ledge}.
		
		For the Type II case we have $r<s$ and $\comp a_r>\comp a_s$. Then there is an entry $j'$ of $T|_{\D(\comp a)}$ immediately right of $j$, with $i<j'\ne 1$.  In particular $j'\ne 1$, so the corresponding cell is in $\D(\comp b')$ as well, where $\comp a'_r\le \comp a'_s\le\comp a_s$. So there is an entry $i'$ immediately right of $i$ in $T|_{\D(\comp a)}$ as well, with $i'\le i<j'$. This argument repeats indefinitely by replacing $i,j$ with $i',j'$ so we must reject the existence of the Type II triple. In this case $T|_{\D(\comp a')}$ is an SSKT.
		
		If instead $\D(\comp a')= \D(\comp a)\cup B(y)$ then in addition to ruling out co-inversion triples, we must ensure none of the new 1 entries from $B(y)$ attack any existing 1 entries in $T|_{\D(\comp a)}$, which we deal with first. Suppose we do have such a pair of attacking cells $z,w$ with $z\in B(y)\subset S$ and $w\in \D(\comp a)$. This pair satisfies (b) and (c) in Definition \ref{def:snake-attack}. Since $w\notin S$, any cell immediately left of $w$ is also not in $S$ so (d) is satisfied as well. Since $z$ is in a column strictly right of $x$, or above $x$ in the same column, satisfying (a) would contradict our choice of $x$, so we must take (a) to be false. The first way this may occur is if $w$ lies below $z$ in the same column. If $y=z$ then $x$ would be in position to attack $w$ contradicting our choice of $y$, so the cell $z'$ immediately left of $z$ must be in $B(y)\subset S$. This makes $z',w$ an $S$-attack which contradicts our choice of $x$. Then (a) must fail by having $w$ in the column immediately left of $z$, necessarily in a row $s$ above the row $r$ of $z$. We see $\comp a_r\le\comp a_s$ as $z$ but not $w$ is in $S$. However $T|_{\D(\comp a)}(w)=1$ which is therefore less than the entry in the same column of row $r$ contradicting Lemma \ref{lem:ledge}. Thus, $T|_{\D(\comp a')}$ is non-attacking.
		
		Now suppose once again that there exists a co-inversion triple $i<j<k$ in $T|_{\D(\comp a')}$, say with $i,k$ in row $r$ and $j$ in row $s$. If it is a Type I triple with $r<s$ and $\comp a'_r>\comp a'_s$ then we must have $j>1$ so the corresponding cell is in $\D(\comp a)$. Since the entries do not form a co-inversion triple in $T|_{\D(\comp a)}$, we must either have $\comp a_r\le\comp a_s$ or the cell containing $i$ not in $\D(\comp a)$, which implies $\comp a_r\le\comp a_s$ anyway. Then $j<k$ contradicts Lemma \ref{lem:ledge}. 
		
		Instead assume we have Type II triple with $r>s$ and $\comp a'_r\ge\comp a'_s$. Similar to before, we must have $\comp a_r<\comp a_s$ to avoid a co-inversion triple in $T|_{\D(\comp a)}$. That is, $\comp a_r<\comp a_s= \comp a'_s\le\comp a'_r$ with $B(y)$ in row $r$. This implies $(\comp a_s,r)\in B(y)$. Additionally, $\comp a\preceq\comp b$ implies $\comp a_s\le\comp a'_r\le \comp b_r<\comp b_s$ so $(\comp a_s+1,s)\in S$. Then $(\comp a_s,r)$ and $(\comp a_s+1,s)$ create an $S$-attack which contradicts the choice $x$. So the Type II co-inversion triple is impossible. We conclude $T|_{\D(\comp a')}$ is an SSKT which completes the proof.
	\end{proof}

	\begin{lemma}\label{lem:attack-invariant}
		In the setting of Definition \ref{def:involution}, a pair of cells $z,w\in\D(\comp b)$ is an $S$-attack if and only if it is an $S'$-attack.
	\end{lemma}
	\begin{proof}
		Recall that $S'$ is a snake and $T\in\mathcal G(S')$ by Lemmas \ref{lem:inv-snake} and \ref{lem:inv-image}. Suppose the pair $(z,w)$ is an $S$-attack. Conditions (a) and (c) in Definition \ref{def:snake-attack} don't depend on $S$ so are satisfied for $S'$. By choice of $x$ we must have $z\notin B(y)$ so $z\in S'$. Suppose $w$ is in the column right of $z$. We know $x$ must lie in a column weakly right of $z$, and if they share a column then $x$ is weakly above $z$. Then $y$ is in a row strictly above $w$ or a column weakly right of $w$. The cell immediately left of $w$ is then not in $B(y)$, and since it cannot be in $S$ by (d) it is not in $S'$ either. Then (d) is satisfied $S'$ so $z,w$ is an $S'$-attack.
		
		Conversely assume $z,w$ is an $S'$-attack. Again, conditions (a) and (c) are already satisfied for $S$. To show $z\in S$ we must show $z\notin B(y)$. Assuming $z\in B(y)$ to the contrary, then $S'=S\cup B$ and $w\in S$ else $T|_{\D(\comp b)\setminus S}$ violates the non-attacking condition. Note that the cell immediately left of $w$ is not contained in $S'$, hence $S$, as it would violate either (d) in Definition \ref{def:snake-attack} or (3) in Definition \ref{def:snake}. Since $S$ is weakly connected, it must contain a cell $v$ attacking $w\ne(1,1)$ from the bottom or left so that $(v,w)$ is an $S$-attack. We cannot have $x$ in a column strictly left of $v$ or strictly right of $z\in B(y)$ so $x$ shares a column with $v$ or $w$. 
		
		If $x$ shares a column with $w$, this forces $z$ into the same column. We must have that $z=y$ with $x$ strictly below $y=z$ strictly below $w$. Since $z\notin S$ and $w\in S$, there must exist a cell $z'$ immediately right of $z$ in $\D(\comp b)$ to satisfy (2) in Definition \ref{def:snake}. But now the $S$-attack $(w,z')$ contradicts the choice of $x$. 
		
		Now suppose that instead of sharing a column with $w$, $x$ shares a column immediately left of $w$ with $v$. By choice of $x$, $x$ lies above $v$. Then $(x,w)$ is an $S$-attack so by choice of $y$, $y$ lies below $w$ in the same column. Since $w$ but not $y$ is in $S$, (2) in Definition \ref{def:snake} requires there to exist a cell $y'$ immediately right of $y$. Now $(w,y')$ is an $S$-attack that contradicts the choice of $x$. Therefore it must be the case that $z\notin B(y)$ so $z\in S$.
		
		Finally to demonstrate (d) in Definition \ref{def:snake-attack}, suppose $w$ is in the column right of $z$. Let $w'$ be the cell immediately left of $w$, which cannot be in $S'$ since $z,w$ is an $S'$-attack. If $w'$ were in $S$ then $w',z$ would be an $S$-attack forcing $x$ to lie in a column strictly right of $w'$, or above $w'$ in the same column. This would prevent $w'$ from being in $B(y)$ and therefore contradict $w'\notin S'$. We have thus shown that $z,w$ is an $S$-attack.
	\end{proof}

	We can now put everything together to see that $\iota$ is a sign-reversing involution.
	\begin{theorem}\label{thm:inv}
		Let $\comp b$ be a weak composition with $\comp b_1>0$. Then $\iota:\mathcal F\to \mathcal F$ is a well-defined involution that reverses the sign of the snake.
	\end{theorem}
	\begin{proof}
		Suppose $(S',T)=\iota(S,T)$. We have $(S',T)\in\mathcal F$ by Lemma \ref{lem:inv-image}. By Lemma \ref{lem:attack-invariant}, an $S$-attack is equivalent to an $S'$-attack and it is therefore apparent from the definition that $\iota(S',T)=(S,T)$. From Lemma \ref{lem:sym-diff} we see that $S$ and $S'$ differ in exactly one row in which exactly one of $S$ and $S'$ contain cells. Thus $\sgn(S)=-\sgn(S')$.
	\end{proof}
	
	The involution $\iota$ is only defined when $\comp b$ is a weak composition with $\comp b_1>0$. Before proving Theorem \ref{thm:snakes} we need one more lemma to sidestep this restriction, and some notation.
	
	Suppose we have a set $I\subset [n]$ with elements $i_1<\cdots<i_\ell$. Define the set of weak compositions 
	\[
	C_{I,k}=\{\comp a\mid |\comp a|=k \text{ and if } \comp a_i> 0 \text{ then } i\in I \}.
	\]
	For any $\comp a\in C_{I,k}$, and any other subset $J\subset [n]$ with the same number of elements $j_1<\cdots<j_\ell$, we let $\comp a^J$ denote the weak composition in $C_{J,k}$ given by $\comp a^J_{j_m}=\comp a_{i_m}$.

	\begin{lemma}\label{lem:iso-subspaces}
		Let $I,J\subset \{1,\ldots,n\}$ with elements $i_1<\cdots<i_\ell$ and $j_1<\cdots <j_\ell$ respectively. The following diagram commutes.
		\begin{center}
			\begin{tikzcd}
				\operatorname{Span}\{ \h_{\comp a}\mid \comp a\in C_{I,k} \} \arrow[r,shift left, "\mathrm{id}"] \arrow[d, shift left, "\h_{\comp a^J}"]
				& \operatorname{Span}\{ \key_{\comp a}\mid \comp a\in C_{I,k}\} \arrow[l,shift left, "\mathrm{id}"] \arrow[d, shift left, "\key_{\comp a^J}"] \\
				\operatorname{Span}\{ \h_{\comp a}\mid \comp a\in C_{J,k} \} \arrow[r,shift left, "\mathrm{id}"] \arrow[u, shift left, "\h_{\comp a^I}"]
				& \operatorname{Span}\{ \key_{\comp a}\mid \comp a\in C_{J,k}\} \arrow[l,shift left, "\mathrm{id}"] \arrow[u, shift left, "\key_{\comp a^I}"]
			\end{tikzcd}
		\end{center}
	\end{lemma}
	\begin{proof}
		Any entry in the first column of a reverse SSAF $T$ must match its row index. Then $\wt(T)_m=0$ implies $\sh(T)_m=0$. By Theorem \ref{thm:key-expand} we can therefore see 
		\[
		\operatorname{Span}\{ \h_{\comp a}\mid \comp a\in C_{I,k} \}\subset\operatorname{Span}\{ \key_{\comp a}\mid \comp a\in C_{I,k} \}
		\]
		and restricting to any complete homogeneous subspace shows they coincide due to dimension.
		
		For $\comp a\in C_{I,k}$ we have $(\comp a^{J})^I_{i_m}=\comp a^{J}_{j_m}=\comp a_{i_m}$ so we see that $\h_{\comp a}\mapsto \h_{\comp a^J}$ and $\key_{\comp a}\mapsto \key_{\comp a^J}$ are indeed isomorphisms with the specified inverses. By Theorem \ref{thm:key-expand} we have
		\[
		\h_{\comp a^J}=\sum_{\substack{T\in\rSSAF\\\wt(T)=\comp a^J}} \key_{\sh(T)}
		\]
		so we are done once we show
		\[
		\sum_{\substack{T\in\rSSAF\\\wt(T)=\comp a^J}} \key_{\sh(T)}=\sum_{\substack{T\in\rSSAF\\\wt(T)=\comp a}} \key_{\sh(T)^J}.
		\]
		
		Indeed, if $T$ is a reverse SSAF of weight $\comp a$, we obtain a new reverse SSAF $T^J$ of weight $\comp a^J$ and shape $\sh(T)^J$ by setting
		\[
		T^J(c,j_r)= j_s
		\]
		whenever $T(c,i_r)=i_s$. Entries in the first column still match their row index. Moreover for a pair of cells $(c,i_r),(d,i_s)$ in $T$, the corresponding pair of cells $(c,j_r),(d,j_s)$ in $T^J$ maintains the same relative ordering of rows, columns, and entries. So $T^J$ will in fact remain a reverse SSAF. The inverse map $T\mapsto T^I$ is defined symmetrically, and we are done.
	\end{proof}

	\begin{proof}[Proof of Theorem \ref{thm:snakes}]
		The statement is trivial for $\comp b=(0,0,\ldots)$ so we take $\comp b$ to have at least one nonzero part and proceed under the assumption that the expansion holds for any weak composition with strictly fewer nonzero parts. We first assume $\comp b_1> 0$.
		
		If $S$ is a special snake of $\D(\comp b)$ then it contains all of the first row so $\D(\comp b)\setminus S$ has a shape with strictly fewer nonzero parts than $\comp b$. Therefore using the inductive hypothesis
		\begin{align*}
			\sum_{\sh(U)=\comp b}\sgn(S)\h_{\wt(U)}&=\sum_{S}\sgn(S)h_{(|S|,0,0,\ldots)} \sum_{\sh(U')=\sh(\D(\comp b)\setminus S)}\sgn(U')\h_{\wt(U')}\\
			&=\sum_{S}\sgn(S)x_1^{|S|} \key_{\sh(\D(\comp b)\setminus S)}\\
			&=\sum_{S}\sgn(S)x_1^{|S|}\sum_{\substack{T\in\SSKT\\ \sh(T)=\sh(\D(\comp b)\setminus S)}}x^{\wt(T)}\\
			&=\sum_{\substack{(S,T)\\ T\in \mathcal G(S)}}\sgn(S)x^{\wt(T)}
		\end{align*}
		where the special snake tabloids $U$ of $\D(\comp b)$ decompose into a special snake $S$ of $\D(\comp b)$ and a special snake tabloid $U'$ of $\D(\comp b)\setminus S$.
		
		It is now our goal to show
		\begin{equation}\label{eq:snake-to-key}
			\sum_{\substack{(S,T)\\ T\in \mathcal G(S)}}\sgn(S)x^{\wt(T)}=\sum_{\substack{T \in \SSKT\\ \sh(T)=\comp b}}x^{\wt(T)}=\key_{\comp b}.
		\end{equation}
		
		The snake $S$ is weakly connected, and for cells in distinct rows to be weakly connected is for them to be attacking. Then if $T\in \mathcal G(S)$ is an SSKT, $S$ can be nothing but the first row of $\D(\comp b)$ and we have $\sgn(S)=1$. Conversely, for $T\in\SSKT$ we have $T\in \mathcal G(S)$ if we take $S$ to be the first row of $\D(\comp b)$. This is to say 
		\[
		\sum_{\substack{(S,T)\\ T\in \mathcal G(S)}}\sgn(S)x^{\wt(T)}=\sum_{\substack{T \in \SSKT\\ \sh(T)=\comp b}}x^{\wt(T)}+\sum_{(S,T)\in \mathcal F}\sgn(S)x^{\wt(T)}.
		\]
		The last summation is zero by Theorem \ref{thm:inv} so \eqref{eq:snake-to-key} follows.
		
		We now drop the assumption that $\comp b_1>0$. Take $I=\{i_1<\cdots< i_\ell\}$ to be the nonzero indices of $\comp b$. Let $J=\{j_1<\cdots<j_\ell\}$ be any subset of $[n]$ of the same size with $j_1=1$. Then $\comp b^J_1=\comp b_{i_1}>0$ and $\comp b^J$ has the same number of nonzero parts as $\comp b$. The special snake tabloid expansion is therefore valid for $\comp b^J$. Using Lemma \ref{lem:iso-subspaces} we have
		\[
		\key_{\comp b}=\key_{(\comp b^J)^I}=\sum_{\sh(U)=\comp b^J}\sgn(U)\h_{\wt(U)^I}.
		\]
		Given a special snake tabloid $U=(S_1,\ldots,S_n)$ of shape $\comp b^J$ we obtain a special snake tabloid $U'=(S'_1,\ldots,S'_n)$ of shape $\comp b$ and weight $\wt(S)^I$ by requiring that $(c,i_r)\in U'_{i_s}$ if and only if $(c,j_r)\in U_{j_s}$. Given a pair of cells $(c,j_r),(d,j_s)$ in $U$, the corresponding pair of cells $(c,i_r),(d,i_s)$ in $U'$ maintains the same relative order of their row, column, and snake indices. So $U'$ indeed remains a special snake tabloid, and the inverse map from special snake tabloids of shape $\comp b$ to tabloids of shape $\comp b^J$ is defined symmetrically. Additionally, each $U_{j_m}$ contains cells in the same number of rows as does $U'_{i_m}$, so $\sgn(U)=\sgn(U')$. We see that
		\[
		\key_{\comp b}=\sum_{\sh(U)=\comp b^J}\sgn(U)\h_{\wt(U)^I}=\sum_{\sh(U)=\comp b}\sgn(U)\h_{\wt(U)}
		\]
		completing the inductive step and the proof.
	\end{proof}


	%
	%
	
	\bibliographystyle{amsalpha} 
	\bibliography{FlaggedKostka}

\providecommand{\bysame}{\leavevmode\hbox to3em{\hrulefill}\thinspace}
\providecommand{\MR}{\relax\ifhmode\unskip\space\fi MR }
\providecommand{\MRhref}[2]{%
  \href{http://www.ams.org/mathscinet-getitem?mr=#1}{#2}
}
\providecommand{\href}[2]{#2}
\begin{thebibliography}{AABE23}

\bibitem[AABE23]{AABE23}
Sam Armon, Sami Assaf, Grant Bowling, and Henry Ehrhard, \emph{Kohnert's rule
  for flagged schur modules}, Journal of Algebra \textbf{617} (2023), 352--381.

\bibitem[AE15]{AE15}
O.~Azenhas and A.~Emami, \emph{An analogue of the
  {R}obinson-{S}chensted-{K}nuth correspondence and non-symmetric {C}auchy
  kernels for truncated staircases}, European J. Combin. \textbf{46} (2015),
  16--44.

\bibitem[AS18]{ASch18}
Sami Assaf and Anne Schilling, \emph{A {D}emazure crystal construction for
  {S}chubert polynomials}, Algebraic Combinatorices \textbf{1} (2018),
  225--247.

\bibitem[AS22]{AS22}
Sami Assaf and Dominic Searles, \emph{Kohnert polynomials}, Exp. Math.
  \textbf{31} (2022), no.~1, 93--119.

\bibitem[Ass18]{Ass18}
Sami Assaf, \emph{Nonsymmetric {M}acdonald polynomials and a refinement of
  {K}ostka--{F}oulkes polynomials}, Trans. Amer. Math. Soc. \textbf{370}
  (2018), no.~12, 8777--8796.

\bibitem[Av22]{AW22}
Sami Assaf and Stephanie {van Willigenburg}, \emph{Skew key polynomials and a
  generalized littlewood–richardson rule}, European Journal of Combinatorics
  \textbf{103} (2022), 103518.

\bibitem[BB93]{BB93}
Nantel Bergeron and Sara Billey, \emph{R{C}-graphs and {S}chubert polynomials},
  Experiment. Math. \textbf{2} (1993), no.~4, 257--269.

\bibitem[BS98]{BS98}
Nantel Bergeron and Frank Sottile, \emph{Schubert polynomials, the {B}ruhat
  order, and the geometry of flag manifolds}, Duke Math. J. \textbf{95} (1998),
  no.~2, 373--423.

\bibitem[CK18]{CK18}
Seung-Il Choi and Jae-Hoon Kwon, \emph{Lakshmibai-{S}eshadri paths and
  non-symmetric {C}auchy identity}, Algebr. Represent. Theory \textbf{21}
  (2018), no.~6, 1381--1394.

\bibitem[Dem74a]{Dem74a}
Michel Demazure, \emph{D\'esingularisation des vari\'et\'es de {S}chubert
  g\'en\'eralis\'ees}, Ann. Sci. \'Ecole Norm. Sup. (4) \textbf{7} (1974),
  53--88, Collection of articles dedicated to Henri Cartan on the occasion of
  his 70th birthday, I.

\bibitem[Dem74b]{Dem74}
\bysame, \emph{Une nouvelle formule des caract\`eres}, Bull. Sci. Math. (2)
  \textbf{98} (1974), no.~3, 163--172.

\bibitem[ER90]{ER90}
{\"O}mer E\u{g}ecio\u{g}lu and Jeffrey~B. Remmel, \emph{A combinatorial
  interpretation of the inverse kostka matrix}, Linear and Multilinear Algebra
  \textbf{26} (1990), no.~1-2, 59--84.

\bibitem[FL09]{FL09}
Amy~M. Fu and Alain Lascoux, \emph{Non-symmetric {C}auchy kernels for the
  classical groups}, J. Combin. Theory Ser. A \textbf{116} (2009), no.~4,
  903--917.

\bibitem[Ful97]{Ful97}
William Fulton, \emph{Young tableaux}, Londn Mathematical Society Student
  Texts, no.~35, Cambridge University Press, Cambridge, 1997.

\bibitem[HHL08]{HHL08}
J.~Haglund, M.~Haiman, and N.~Loehr, \emph{A combinatorial formula for
  nonsymmetric {M}acdonald polynomials}, Amer. J. Math. \textbf{130} (2008),
  no.~2, 359--383.

\bibitem[HMR13]{HMR13}
James Haglund, Sarah Mason, and Jeffrey Remmel, \emph{Properties of the
  nonsymmetric {R}obinson-{S}chensted-{K}nuth algorithm}, J. Algebraic Combin.
  \textbf{38} (2013), no.~2, 285--327.

\bibitem[KK02]{KK02}
Mikhail Kogan and Abhinav Kumar, \emph{A proof of {P}ieri's formula using the
  generalized {S}chensted insertion algorithm for rc-graphs}, Proc. Amer. Math.
  Soc. \textbf{130} (2002), no.~9, 2525--2534.

\bibitem[Koh91]{Koh91}
Axel Kohnert, \emph{Weintrauben, {P}olynome, {T}ableaux}, Bayreuth. Math. Schr.
  (1991), no.~38, 1--97, Dissertation, Universit{\"a}t Bayreuth, Bayreuth,
  1990.

\bibitem[Las03]{Las03}
Alain Lascoux, \emph{Double crystal graphs}, Studies in memory of Issai Schur,
  Birkhäuser, Boston, 2003, pp.~95--114.

\bibitem[LLT97]{LLT97}
Alain Lascoux, Bernard Leclerc, and Jean-Yves Thibon, \emph{Ribbon tableaux,
  hall–littlewood functions, quantum affine algebras, and unipotent
  varieties}, Journal of Mathematical Physics \textbf{38} (1997), no.~2,
  1041--1068.

\bibitem[LS82]{LS82}
Alain Lascoux and Marcel-Paul Sch{\"u}tzenberger, \emph{Polyn\^omes de
  {S}chubert}, C. R. Acad. Sci. Paris S\'er. I Math. \textbf{294} (1982),
  no.~13, 447--450.

\bibitem[LS90]{LS90}
\bysame, \emph{Keys \& standard bases}, Invariant theory and tableaux
  ({M}inneapolis, {MN}, 1988), IMA Vol. Math. Appl., vol.~19, Springer, New
  York, 1990, pp.~125--144.

\bibitem[Mac88]{Mac88}
I.~G. Macdonald, \emph{A new class of symmetric functions}, Actes du 20e
  Seminaire Lotharingien \textbf{372} (1988), 131--171.

\bibitem[Mac96]{Mac96}
\bysame, \emph{Affine {H}ecke algebras and orthogonal polynomials},
  Ast\'erisque (1996), no.~237, Exp.\ No.\ 797, 4, 189--207, S\'eminaire
  Bourbaki, Vol. 1994/95.

\bibitem[Mas08]{Mas08}
Sarah Mason, \emph{A decomposition of schur functions and an analogue of the
  robinson-schensted-knuth algorithm}, SÉM. LOTHAR. COMBIN \textbf{57} (2008),
  B57e, 24 p.

\bibitem[Mas09]{Mas09}
Sarah Mason, \emph{An explicit construction of type {A} {D}emazure atoms}, J.
  Algebraic Combin. \textbf{29} (2009), no.~3, 295--313.

\bibitem[Mur37]{Mur37}
F.~D. Murnaghan, \emph{On the representations of the symmetric group}, Amer. J.
  Math. \textbf{59} (1937), no.~3, 437--488.

\bibitem[Nak40]{Nak40}
Tadasi Nakayama, \emph{On some modular properties of irreducible
  representations of a symmetric group, i}, Jpn. J. Math. \textbf{17} (1940),
  165--184.

\bibitem[Opd95]{Opd95}
Eric~M. Opdam, \emph{Harmonic analysis for certain representations of graded
  {H}ecke algebras}, Acta Math. \textbf{175} (1995), no.~1, 75--121.

\bibitem[Pun16]{P16}
Ying~Anna Pun, \emph{On decomposition of the product of {D}emazure atoms and
  {D}emazure characters}, Ph.D. thesis, University of Pennsylvania, 2016.

\bibitem[RS98]{RS98}
Victor Reiner and Mark Shimozono, \emph{Percentage-avoiding, northwest shapes
  and peelable tableaux}, J. Combin. Theory Ser. A \textbf{82} (1998), no.~1,
  1--73.

\bibitem[San00]{San00}
Yasmine~B. Sanderson, \emph{On the connection between {M}acdonald polynomials
  and {D}emazure characters}, J. Algebraic Combin. \textbf{11} (2000), no.~3,
  269--275.

\bibitem[Sot96]{Sot96}
Frank Sottile, \emph{Pieri's formula for flag manifolds and {S}chubert
  polynomials}, Ann. Inst. Fourier (Grenoble) \textbf{46} (1996), no.~1,
  89--110.

\bibitem[Sta99]{Sta99}
Richard~P. Stanley, \emph{Enumerative combinatorics}, vol.~2, Cambridge
  University Press, Cambridge, 1999.

\end{thebibliography}

\end{document}